\documentclass[11pt,reqno]{amsart}    

\usepackage[margin=1in]{geometry}
\usepackage{amsfonts,amsmath,amssymb,enumerate}    
\usepackage{amsthm}    
\usepackage{mathrsfs}  
\usepackage{setspace}
\usepackage{verbatim}
\onehalfspace
\usepackage{tikz}    
\usetikzlibrary{arrows,matrix}
\usetikzlibrary{decorations.markings,shapes.geometric,shapes.misc}    

\numberwithin{equation}{section}    

\theoremstyle{plain}% default
\newtheorem{thm}{Theorem}[section]
\newtheorem{lem}[thm]{Lemma}
\newtheorem{prop}[thm]{Proposition}
\newtheorem{cor}[thm]{Corollary}
\newtheorem{conj}[thm]{Conjecture}
\theoremstyle{definition}
\newtheorem{defn}[thm]{Definition}

\theoremstyle{remark}
\newtheorem{rem}[thm]{Remark}

\newtheorem*{rem*}{Remark}
\newtheorem*{ack}{Acknowledgements}

\newcommand{\bs}{\boldsymbol}

\newcommand{\be}{\begin{equation}}    
\newcommand{\ee}{\end{equation}}    
\newcommand{\beu}{\begin{equation*}}    
\newcommand{\eeu}{\end{equation*}}    
\newcommand{\bea}{\begin{eqnarray}}    
\newcommand{\eea}{\end{eqnarray}}    
\newcommand{\beaa}{\begin{eqnarray*}}    
\newcommand{\eeaa}{\end{eqnarray*}}    
\newcommand{\bmx}{\begin{pmatrix}}    
\newcommand{\emx}{\end{pmatrix}}

\newcommand{\g}{{\mathfrak g}}    
\newcommand{\h}{{\mathfrak h}}

\newcommand{\mf}{\mathfrak}
\newcommand{\mc}{\mathcal}    
\newcommand{\al}{{\alpha}}    
\newcommand{\gh}{{\widehat \g}}

\newcommand{\alf}{{\textstyle{\frac{1}{2}}}}

\newcommand{\nn}{\nonumber}

\newcommand{\8}{{\infty}}

\newcommand{\eps}{\epsilon}

\newcommand{\rank}{{\rm rank}}

\newcommand{\Z}{{\mathbb Z}}
\newcommand{\N}{{\mathbb N}}
\newcommand{\C}{{\mathbb C}}

\newcommand{\Q}{{\mathcal Q}}

\newcommand{\id}{{\mathrm{id}}}

\newcommand{\uq}{{U_q}}

\newcommand{\uqslth}{{\uq(\widehat{\mathfrak{sl}}_2})}
\newcommand{\uqslt}{{\uq(\mf{sl}_2)}}    
\newcommand{\uqgh}{{\uq(\widehat\g)}}    

\newcommand{\uqg}{{\uq(\g)}}

\newcommand{\A}{\mathbf A}

\newcommand{\goi}[2]{=}

\newcommand{\Ih}{\widehat I}

\newcommand{\on}{.}    

\newcommand{\groth}[1]{{\mathrm{Groth}(#1)}}    
\newcommand{\Cx}{\mathbb C^\times}
\newcommand{\qnum}[1]{\left[ #1\right]_q}
\renewcommand{\binom}[2]{\begin{bmatrix} #1 \\ #2 \end{bmatrix}}
\newcommand{\qbinom}[2]{\binom{#1}{#2}_q}
    
\newcommand{\btp}{\begin{tikzpicture}[baseline=0pt,scale=0.9,line width=0.25pt]}    
\newcommand{\etp}{\end{tikzpicture}}    

\newcommand{\Roff}{\color{black}}

\renewcommand{\L}{L}
\newcommand{\atp}[1]{}

\newcommand{\path}{\longrightarrow}

\DeclareMathOperator{\wt}{wt}
\DeclareMathOperator{\hgt}{height}
\DeclareMathOperator{\ev}{ev}

\DeclareMathOperator{\Span}{span}

\DeclareMathOperator{\codim}{codim}

\DeclareMathOperator{\Ob}{Ob}
\newcommand{\cat}{\mathcal O}
\newcommand{\cinv}{\varphi}
\newcommand{\la}{\lambda}
\newcommand{\om}{\omega}

\author{E. Mukhin} 
\address{\vspace{-.15cm} Department of Mathematical Sciences, 402
  N. Blackford St, LD 270, IUPUI, Indianapolis, IN 46202, USA.  }
\email{mukhin@math.iupui.edu}

\author{C. A. S. Young}

\address{\vspace{-.15cm} School of Physics, Astronomy and Mathematics, University of Hertfordshire, College Lane, Hatfield AL10 9AB, UK.}  \email{charlesyoung@cantab.net}
\begin{document}
\title{Affinization of category $\cat$ for quantum groups}
\begin{abstract} 
Let $\g$ be a simple Lie algebra. We consider the category $\hat\cat$ of those modules over the affine quantum group $\uqgh$ whose $\uqg$-weights have finite multiplicity and lie in a finite union of cones generated by negative
roots.
We show that many properties of the category of the finite-dimensional representations naturally extend to the category $\hat\cat$. In particular, we develop the theory of $q$-characters and 
define the minimal affinizations of parabolic Verma modules. In types ABCFG we classify these minimal affinizations and conjecture a Weyl denominator type formula for their characters.
\end{abstract}
\maketitle
\section{Introduction}
Let $\g$ be a simple Lie algebra and $q\in \Cx$ transcendental. In this paper we consider the category $\hat\cat$ of modules over the affine quantum group $\uqgh$ such that after the restriction to $\uqg$ the dimensions of the weight spaces are finite, and the set of non-trivial weights belongs to a finite union of cones generated by negative
%s of the simple 
roots. This category was originally defined in \cite{HernandezFusion}. It is a tensor category which includes the finite-dimensional modules. The simple objects in $\hat\cat$ are highest weight $\uqgh$-modules with highest $\ell$-weights given by arbitrary sets of rational functions $(f_i)_{i\in I}$ with the property $f_i(0)f_i(\infty)=1$, $i\in I$, $I$ being the set of nodes of the Dynkin diagram of $\g$ (see Theorem \ref{thm1} below). Our motivation for the study of $\hat\cat$ is twofold.  

\medskip

First, many results from the category of finite-dimensional $\uqgh$-modules can be easily extended to the much richer category $\hat\cat$. For example, we have in $\hat\cat$ the classification of irreducibles by highest $\ell$-weights, and the notions of fundamental modules, Kirillov-Reshetikhin modules and minimal affinizations. In type $sl_2$, the irreducible modules are tensor products of evaluation modules. We define a theory of $q$-characters which gives an injective ring homomorphism from the Grothendieck ring of $\hat\cat$ to a certain formal ring possessing many 
properties which allow us to study it combinatorially.

Second, we are trying to find a new way to study the minimal affinizations of the finite-dimensional modules. Minimal affinizations, which are analogs of the evaluation modules that exist only in type A, received a lot of attention, see \cite{CP94,Cminaffrank2,CPminaffBCFG,CPminaffireg,CPminaffADE,HernandezMinAff,Moura,MY1,MY2} but are still poorly understood in general.  In the non-affine setting, important information about the finite-dimensional modules comes from the study of Verma modules, which have a much simpler structure. Inspired by this idea we initiate the study of minimal affinizations of Verma modules, which naturally leads us to the category $\hat\cat$.

\medskip

We establish the foundations of the theory of the category $\hat\cat$, for the most part 
modifying the well-known methods initially developed by many authors for the finite-dimensional modules. As one notable exception, we give a proof of the classification of the minimal affinizations somewhat different from the classical papers \cite{Cminaffrank2,CPminaffBCFG}; see Theorem \ref{lathm}.  We use the theory of $q$-characters and treat types ABCFG simultaneously. In these types a minimal affinization is not only a \emph{minimal} element with respect to the partial order defined in \cite{Cminaffrank2}, but the \emph{least} element. In this paper we do not the consider the types, D and E, whose diagrams have a trivalent node.

\medskip

Our main finding is that the minimal affinizations of the generic parabolic Verma modules (and many other modules) considered as $\uqg$-modules indeed have a simple character similar to the Weyl denominator. For example, if $\la=\sum_{i\in I} \la_i\om_i$ is a $\g$-weight written in terms of the fundamental weights and none of $\la_i$ is an integer, we conjecture that the character of the minimal affinization of the Verma module with highest weight $\la$ is given by
$$
\chi_\la=e^\la\prod_{\al\in \Delta^+}\frac{1}{(1-e^{\al})^{m_\al}}\ ,
$$
where $\Delta^+$ is the set of positive roots of $\g$ and for a positive root
$\al=\sum_{i\in I}\al_i\omega_i$ we define  $m_\al=\max\limits_{i\in I}\{\al_i\}\in\Z_{\geq 1}$. 

This formula and many similar formulae, see Conjecture \ref{charconj},
were found and partially checked with the help of a computer based on the use of the algorithm of \cite{FM}. We give proofs only in some special cases, e.g. in types $A_n$, $B_2$, based on known results for finite-dimensional modules, but the simplicity of the answer suggests that a general proof may be not very difficult. 

\medskip

We would like to acknowledge the paper \cite{HJ} where the authors studied the
stable limits of the Kirillov-Reshetikhin modules, which are minimal affinization of finite-dimensional modules with highest weights $\la=n\omega_i$, as $n\to\infty$. Those limits are representations of an algebra which is slightly different from the standard quantum affine group. Instead of going to that limit we study the analytic continuation with respect to $n$, see \S\ref{acsec}. That, in particular, allows us to stay with the standard quantum affine group.

\medskip

The paper is structured as follows.
After summarizing background material in \S2, in \S3 we define the category $\hat\cat$, classify its simple objects (Theorem \ref{thm1}), and develop the theory of $q$-characters for $\hat\cat$. 
We also briefly discuss analytic continuation (\S\ref{acsec}) and the restricted duals of objects in $\hat\cat$ (\S\ref{dualsec}). 
In \S4 we consider the case $\g=\mf{sl}_2$ and give a description of the simple objects in $\hat\cat$ in terms of tensor products of evaluation modules. 
Affinizations, and in particular minimal and least affinizations, are introduced in \S5. Theorem \ref{lathm} classifies least affinizations in types ABCFG.
Finally, \S6 contains a series of conjectural formulae for the $\uqg$-characters of least affinizations of parabolic Verma modules, and of certain other representations.  
 
\begin{ack} 
From October 2010 until December 2011, the research of CASY was funded by the EPSRC, grant number EP/H000054/1. The research of EM is supported by the NSF, grant number DMS-0900984.
Computer programs to calculate $q$-characters were written in FORM \cite{FORM}.
\end{ack}

\section{Background}\label{qcharsec}
\subsection{Cartan data}
Let $\g$ be a complex simple Lie algebra of rank $N$ and $\h$ a Cartan subalgebra of $\g$. We identify $\h$ and $\h^*$ by means of the invariant inner product $\left<\cdot,\cdot\right>$  on $\g$ normalized such that the square length of the maximal root equals 2.  With $I=\{1,\dots,N\}$, let $\{\alpha_i\}_{i\in I}$ be a set of simple roots, and $\{\alpha^\vee_i\}_{i\in I}$, $\{\omega_i\}_{i\in   I}$, $\{\omega^\vee_i\}_{i\in I}$ the corresponding sets of, respectively, simple coroots, fundamental weights and fundamental coweights. Let $C=(C_{ij})_{i,j\in I}$ be the Cartan matrix. We have 
\be\nn 2 \left< \alpha_i, \alpha_j\right> = C_{ij} \left<   \alpha_i,\alpha_i\right>,\quad 2 \left< \alpha_i, \omega_j\right> = \delta_{ij}\left<\alpha_i,\alpha_i\right>, \quad \left<\omega^\vee_i,\alpha_j\right> = \delta_{ij} = \left<\alpha^\vee_i,\omega_j\right>.\ee 
Let $r^\vee$ be the maximal number of edges connecting two vertices of the Dynkin diagram of $\g$. Thus $r^\vee=1$ if $\g$ is of types A, D or E, $r^\vee = 2$ for types B, C and F and $r^\vee=3$ for $\mathrm G_2$.  Let $r_i= \alf r^\vee \left<\alpha_i,\alpha_i\right>$. The numbers $(r_i)_{i\in   I}$ are relatively prime integers. We set \be\nn D:= \mathrm{diag}(r_1,\dots,r_N),\qquad B := DC;\ee the latter is the symmetrized Cartan matrix, $B_{ij} = r^\vee \left<\alpha_i,\alpha_j\right>$.
 
Let $Q$ (resp. $Q^+$) and $P$ (resp. $P^+$) denote the $\Z$-span (resp. $\Z_{\geq 0}$-span) of the simple roots and fundamental weights respectively. Let $\leq$ be the partial order on $\h^*$ (and in particular on $P$ and $Q$) in which $\lambda\leq \lambda'$ if and only if $\lambda'-\lambda\in Q^+$. If $\eta=\sum_{i\in I} m_i \alpha_i\in Q^+$, define $\hgt(\eta) = \sum_{i\in I} m_i$. 
 
Let $\Delta\subset Q$ be the set of roots of $\g$ and $\Delta^+ = \Delta\cap Q^+$ the set of positive roots.

Let $\gh$ denote the untwisted affine algebra corresponding to $\g$. 
Let $\widehat C= (C_{ij})_{i,j\in \{0\} \cup I}$ be the extended Cartan matrix, $\alpha_0$ be the extra simple root of $\gh$, $r_0= \alf r^\vee \left< \alpha_0,\alpha_0 \right>$, $\widehat D= \mathrm{diag}(r_0,r_1,\dots,r_N)$ and $\widehat B = \widehat D \widehat C$.

Fix a transcendental $q \in \Cx$. For each $i\in I$ let \be q_i:= q^{r_i}.\nn\ee 
Define the $q$-numbers, $q$-factorial and $q$-binomial: \be\nn \qnum n := \frac{q^n-q^{-n}}{q-q^{-1}},\quad \qnum n ! := \qnum n \qnum{n-1} \dots \qnum 1,\quad \qbinom n m := \frac{\qnum n !}{\qnum{n-m} ! \qnum m !}.\ee

\subsection{Quantum Affine Algebras}
The \emph{quantum affine algebra} $\uqgh$ in the Drinfeld-Jimbo realization \cite{Drinfeld1, Jimbo} is the unital associative algebra over $\mathbb C$ with generators $(x_i^\pm)_{i\in\Ih}$, $(k_i^{\pm 1})_{i\in \Ih}$ subject to the relations
\begin{align} \label{chevgens}
  k_ik_i^{-1} = k_i^{-1}k_i &=1,\quad \quad k_ik_j =k_jk_i,\\
  k_ix_j^{\pm}k_i^{-1} &= q^{\pm \widehat B_{ij}}x_j^{\pm},\nn\\ 
[x_i^+ , x_j^-] &= \delta_{ij}\frac{k_i - k_i^{-1}}{q_i -q_i^{-1}},\nn\\
  \sum_{r=0}^{1-C_{ij}}(-1)^r\left[\begin{array}{cc} 1-\widehat C_{ij}\\ r \end{array} \right]_{q_i}
  (x_i^{\pm})^rx_j^{\pm}&(x_i^{\pm})^{1-\widehat C_{ij}-r} =0, \ \ \ \
  i\ne j\nn.
\end{align}
The algebra $\uqgh$ can be endowed with the coproduct, antipode and counit given by
\bea \Delta(k_i) &=& k_i \otimes k_i, \nn\\
\Delta(x^+_i) &=& x^+_i \otimes k_i + 1 \otimes x^+_i,\nn\\ \Delta(x^-_i)
              &=& x^-_i \otimes 1 + k_i^{-1} \otimes x^-_i, \nn \eea 
\be S(x_i^+) =
-x_i^+ k^{-1}_i,\qquad S(x_i^-) = - k_i x_i^-,\qquad S(k_i^{\pm 1}) =
k_i^{\mp 1}, \nn\ee \be \eps(k_i) = 1,\quad \eps(x_i^+) = \eps(x_i^-) = 0,
\nn\ee 
making it into a Hopf algebra.

There exists another presentation of  $\uqgh$, due to Drinfeld \cite{Drinfeld}. In this presentation $\uqgh$ is generated by $x_{i,n}^{\pm}$ ($i\in I$, $n\in\Z$), $k_i^{\pm 1}$ ($i\in I$), $h_{i,n}$ ($i\in I$, $n\in \Z\setminus\{0\}$) and central elements $c^{\pm 1/2}$, subject to the following relations:
\begin{align}
  k_ik_j = k_jk_i,\quad & k_ih_{j,n} =h_{j,n}k_i,\nn\\
  k_ix^\pm_{j,n}k_i^{-1} &= q^{\pm B_{ij}}x_{j,n}^{\pm},\nn\\
 \label{hxpm} [h_{i,n} , x_{j,m}^{\pm}] &= \pm \frac{1}{n} [n B_{ij}]_q c^{\mp
    {|n|/2}}x_{j,n+m}^{\pm},\\ 
x_{i,n+1}^{\pm}x_{j,m}^{\pm} -q^{\pm B_{ij}}x_{j,m}^{\pm}x_{i,n+1}^{\pm} &=q^{\pm
    B_{ij}}x_{i,n}^{\pm}x_{j,m+1}^{\pm}
  -x_{j,m+1}^{\pm}x_{i,n}^{\pm},\nn\\ [h_{i,n},h_{j,m}]
  &=\delta_{n,-m} \frac{1}{n} [n B_{ij}]_q \frac{c^n -
    c^{-n}}{q-q^{-1}},\nn\\ [x_{i,n}^+ , x_{j,m}^-]=\delta_{ij} & \frac{
    c^{(n-m)/2}\phi_{i,n+m}^+ - c^{-(n-m)/2} \phi_{i,n+m}^-}{q_i -
    q_i^{-1}},\nn\\
  \sum_{\pi\in\Sigma_s}\sum_{k=0}^s(-1)^k\left[\begin{array}{cc} s \nn\\
      k \end{array} \right]_{q_i} x_{i, n_{\pi(1)}}^{\pm}\ldots
  x_{i,n_{\pi(k)}}^{\pm} & x_{j,m}^{\pm} x_{i,
    n_{\pi(k+1)}}^{\pm}\ldots x_{i,n_{\pi(s)}}^{\pm} =0,\ \
  s=1-C_{ij},\nn
\end{align}
for all sequences of integers $n_1,\ldots,n_s$, and $i\ne j$, where $\Sigma_s$ is the symmetric group on $s$ letters, and $\phi_{i,n}^{\pm}$'s are determined by equating coefficients of powers of $u$ in the formula 
\begin{equation} \label{phidef} \phi_i^\pm(u) :=
  \sum_{n=0}^{\infty}\phi_{i,\pm n}^{\pm}u^{\pm n} = k_i^{\pm 1}
  \exp\left(\pm(q-q^{-1})\sum_{m=1}^{\infty}h_{i,\pm m} u^{\pm
      m}\right).
\end{equation}
Note that $\phi^+_{i,-n}=\phi^-_{i,n}=0$ for all $n\in \Z_{>0}$, and $\phi^\pm_{i,0} = k_{i}^{\pm 1}$. 

We have $x^\pm_{i,0} = x^\pm_i$ for all $i\in I$. 

The subalgebra of $\uqgh$ generated by $(k_i)_{i\in I}$, $(x^\pm_{i})_{i\in I}$ is a Hopf subalgebra of $\uqgh$ and is isomorphic as a Hopf algebra to $\uqg$, the quantized enveloping algebra of $\g$.
In this way, $\uqgh$-modules restrict to $\uqg$-modules. 
The Cartan involution of $\uqg$ is defined by
\be \cinv(x_i^\pm) = -x_i^\mp, \quad \cinv(k_i^{\pm 1}) = k_i^{\mp 1},\quad i\in I.\label{cinvdef}\ee

We shall need the following quantum-affine analog $\hat\cinv$ of the Cartan involution. By definition, \cite{Cminaffrank2}, $\hat\cinv$ is the algebra automorphism whose action on generators is:
\be \hat\cinv(x_{i,r}^\pm) = -x_{i,-r}^\mp, \quad \hat\cinv(h_{i,r}) = -h_{i,-r}, \quad \hat\cinv(k_i^{\pm1}) = k_i^{\mp1},\quad \hat\cinv(c^{\pm1/2}) = c^{\mp1/2} .\label{cinvhatdef}\ee
Note that
\be \hat\cinv\left(\phi^\pm(u)\right) = \phi^\mp(u^{-1})\label{cinvhatphi}.\ee
%It has the property that
%\be \hat\cinv \circ S = \kappa \circ S \circ \hat\cinv,\ee
%where $\kappa$ is the Hopf algebra automorphism given by
%\be \kappa(x^\pm_{i}) = q_i^{\pm 2} x^\pm_i, \quad \kappa(k_i^{\pm 1}) = k_i^{\pm 1}, \quad i \in \hat I.\ee 

Let $\hat U^\pm\subset \uqgh$ be the subalgebras generated by $(x^\pm_{i,r})_{i\in I, r\in \Z}$, and $U^\pm\subset \uqg$ the subalgebras generated by $(x^\pm_i)_{i\in I}$.
Let $\hat U^0\subset\uqgh$ be the subalgebra generated by $c^{\pm1/2}$, $(k_i)_{i\in I}$ and $(h_{i,r})_{i\in I,r\in \Z_{\neq 0}}$, and $U^0\subset\uqg$ the subalgebra generated by $(k_i)_{i\in I}$. We have the following triangular decompositions of $\uqg$ and of $\uqgh$ \cite{CPbook}: 
\begin{align} \uqgh &= \hat U^- . \hat U^0 . \hat U^+ \label{uqghtriang}\\
              \uqg  &= U^- . U^0 . U^+. \label{uqgtriang}\end{align}

It is known \cite{Dam} that on representations of $\uqgh$ on which $c$ acts as the identity, 
\be \Delta \phi_i^\pm(u) = \phi_i^\pm(u) \otimes \phi_i^\pm(u) \mod \hat U^- \otimes \hat U^+.\label{damiani}\ee

%The \emph{quantum loop algebra} is $\uqlg$ is by definition $\uqgh/(c^{\pm 1/2}-1)$. It inherits from $\uqgh$ the structure of a Hopf algebra.

% In the present paper we are concerned only with finite-dimensional representations of $\uqgh$. As we recall below, for this purpose it actually suffices to work with the \emph{quantum loop algebra} $U_q(L\g) = \uqgh/(c^{\pm 1/2}-1)$.

\section{The category $\hat\cat$}
\subsection{Definition of $\hat\cat$}
Let $\overline{\phantom Q}:\h^* \to (\Cx)^N$ be the surjective homomorphism of abelian groups  such that 
\be \overline{\sum_{i\in I}\lambda_i \omega_i} := (q_1^{\lambda_1},q_2^{\lambda_2},\dots,q_N^{\lambda_N}).\label{bardef}\nn\ee 
By a slight overloading, we use the word \emph{weight} to refer to an $N$-tuple \be\nn\varrho =(\varrho_i)_{i\in I}\in \overline{\h^*}\equiv (\Cx)^N.\ee
Since $q$ is not a root of unity, the restrictions of $\overline{\phantom P}$ to $P$ and in particular to $Q^+$ are injective; let $\overline{P}$ and $\overline{Q^+}$ denote their respective images. Then $\overline{\h^*}$ inherits from $\h^*$ the usual partial order:
\be \varrho \leq \varrho' \Leftrightarrow \varrho^{-1}\varrho' \in \overline{Q^+}.\label{pop}\ee

A $\uqg$-module $V$ is a \emph{weight module} if
\be V = \bigoplus_{\varrho\in \overline{\h^*}} V_\varrho\,\,\,,\qquad\quad 
V_\varrho = \{ v \in V: k_i \on v = \varrho_i v\}.\label{gwts}\ee 
We call $V_\varrho$ the \emph{weight space} of weight $\varrho$, and nonzero elements $v\in V_\varrho$ \emph{weight vectors} of weight $\varrho$. 
We say $\varrho\in \overline{\h^*}$ is a \emph{weight of $V$} if $\dim V_\varrho>0$. 

\begin{defn}\label{catdef} We say a $\uqg$-module $V$ is in category $\cat$ if:
\begin{enumerate}[(i)]
\item\label{fw} $V$ is a weight module all of whose weight spaces are finite-dimensional.
\item\label{fgen} There exist a finite number of weights $\varrho_1, \dots, \varrho_k\in \overline{\h^*}$ such that every weight of $V$ is in $\bigcup_{j=1}^k \left\{\varrho_jx^{-1}: x\in \overline{Q^+}\right\}$. 
\end{enumerate}
\end{defn}

Let us define an \emph{$\ell$-weight module} to be any $\uqgh$-module on which the actions of the generators $(h_{i,r})_{i\in I,r\in \Z_{\neq 0}}$ commute pairwise. 
\begin{prop}\label{c1prop}
Every simple $\uqgh$-module $V$ whose restriction as a $\uqg$-module is in $\cat$ is an $\ell$-weight module. Moreover it can be obtained by twisting, by an automorphism of $\uqgh$, a module in which $c^{1/2}$ acts as the identity. 
\end{prop}
\begin{proof}
Since the invertible central element $c^{1/2}$ acts as a multiple of the identity on any simple module, there exists a $\tau\in \C$ such that $c^{1/2}\on v=\tau v$ for all $v\in V$. Then each weight space $V_{\varrho}$ carries a representation of the 3-dimensional Lie algebra generated by $h_{i,r}$, $h_{j,s}$ and $(c-c^{-1})$. By Definition \ref{catdef} part (\ref{fw}), $V_{\varrho}$ is finite-dimensional. The Weyl algebra $\C[x,p]/\left<xp-px-1\right>$ does not admit finite-dimensional representations. Therefore  $\tau^2-\tau^{-2}=0$. Hence $c-c^{-1}$ acts as zero on $V$. This proves the first part. If $\tau^2=-1$ then the map 
\be c^{1/2}\mapsto \tau^{-1} c^{1/2}, \quad k_i\mapsto k_i,\quad h_{i,r} \mapsto \tau^{-|r|} h_{i,r} , \quad x^\pm_{j,s} \mapsto (\mp 1)^{s} x_{j,s}^\pm\nn\ee
defines an automorphism of $\uqgh$; twisting by it we indeed arrive at a module on which $c^{1/2}$ acts as the identity. On the other hand if $\tau=-1$, we may twist by the automorphism of $\uqgh$ defined by
\be c^{1/2}\mapsto - c^{1/2}, \quad k_i\mapsto k_i,\quad h_{i,r} \mapsto h_{i,r} , \quad x^\pm_{j,s} \mapsto (\mp 1)^{s} x_{j,s}^\pm,\nn\ee
with the same result.
\end{proof}
\begin{defn}\label{hatcatdef}
We say a $\uqgh$-module is in category $\hat\cat$ if its restriction as a $\uqg$-module is in category $\cat$ and $c^{1/2}$ acts as the identity on $V$. 
\end{defn}

Definitions \ref{catdef} and \ref{hatcatdef} were stated in \cite{HernandezFusion}. 

The category $\cat$ is a subcategory of the abelian monoidal category of all $\uqg$-modules.
It is clear that $\cat$ is closed under taking quotients, submodules and finite direct sums, and tensor products. Therefore $\cat$ is an abelian monoidal category. 

Likewise, $\hat\cat$ is an abelian monoidal subcategory of the category of all $\uqgh$-modules.
Every $V\in \Ob\hat\cat$ is an $\ell$-weight module.

\begin{rem}
Because we wish $\cat$ to be closed under tensor products, we do not require that every object $V$ of $\cat$ be finitely generated as a $\uqg$-module.
Similarly, inside our category $\hat\cat$ there is a subcategory consisting of those objects that are finitely generated \emph{as $\uqgh$-modules}. This subcategory contains all simple objects of $\hat\cat$ (these are classified in Theorem \ref{thm1} below) and is strictly smaller than $\hat\cat$. It is an interesting question whether this subcategory is closed under taking tensor products.
\end{rem}

\subsection{Classification of simple objects}
Given $V\in\Ob\hat\cat$, the decomposition (\ref{gwts}) into weight spaces can be refined as follows. An \emph{$\ell$-weight} is any $N$-tuple of sequences of complex numbers 
\be\bs\gamma \equiv (\gamma_{i,\pm r}^\pm)_{i\in I, r\in \Z_{\geq 0}},\nn\ee
such that $\gamma_{i,0}^+\gamma_{i,0}^- = 1$ for every $i\in I$. 
Given an $\ell$-weight $\bs\gamma$ we define its weight to be 
\be \wt(\bs\gamma) := (\gamma_{i,0}^+)_{i\in I} \in \overline{\h^*}.\label{wtdef}\nn\ee
Then for every weight $\varrho$ of $V$ we have, c.f. (\ref{phidef}),
\be V_\varrho = \bigoplus_{\bs\gamma: \wt(\bs\gamma)=\varrho} V_{\bs\gamma}\,,\qquad V_{\bs\gamma} = \{ v \in V : \exists k \in \mathbb Z_{> 0}, \,\, \forall i \in I, r\geq 0 \quad \left(
  \phi_{i,\pm r}^\pm - \gamma_{i,\pm r}^\pm\right)^k \on v = 0 \} \,, \label{lwdecomp} \nn\ee 
where the sum is over all $\ell$-weights of weight $\varrho$. 
We call $V_{\bs\gamma}$ the \emph{$\ell$-weight space} of $\ell$-weight $\bs\gamma$. 
We say $\bs\gamma$ is an \emph{$\ell$-weight of $V$} if $\dim (V_{\bs\gamma}) >0$. If $v\in V_{\bs\gamma}$ is nonzero and moreover $\phi_{i,\pm r}^\pm . v = \gamma_{i,\pm r}^\pm v$ for all $i\in I$, $r\in \Z_{\geq 0}$, then $v$ is called an \emph{$\ell$-weight vector} of $\ell$-weight $\bs\gamma$. Every $\ell$-weight space contains an $\ell$-weight vector. If $v\in V$ is nonzero and $x_{i,r}^+.v=0$ for all $i\in I$, $r\in \Z$, then we say the vector $v$ is \emph{singular}.

We say $V\in\Ob\hat\cat$ is a \emph{highest $\ell$-weight representation} of \emph{highest $\ell$-weight} $\bs\gamma$ if $V=\uqgh.v$ for some singular $\ell$-weight vector $v\in V_{\bs \gamma}$. By (\ref{uqghtriang})  $\dim(V_{\bs\gamma})=1$, so $v$ is unique up to scale; we call it the \emph{highest $\ell$-weight vector} of $V$.

\begin{defn}\label{rdef}
We say an $\ell$-weight $\bs f = (f_{i,\pm r}^\pm)_{i\in I, r\in \Z_{\geq 0}}$ is \emph{rational} if  there is an $N$-tuple of complex-valued rational functions $(f_i(u))_{i\in I}$ of a formal variable $u$ such that, for each $i\in I$, $f_i(u)$ is regular at $0$ and $\8$,  $f_i(0)f_i(\8)=1$, and
\be \sum_{r=0}^\8 f^+_{i,r} u^r = f_i(u) = \sum_{r=0}^\8 f^-_{i,-r} u^{-r} \label{rform}\nn\ee
in the sense that the left- and right-hand sides are the Laurent expansions of $f_i(u)$ about $0$ and $\8$, respectively. 

Let $\mc R$ be the set of rational $\ell$-weights. $\mc R$ forms an abelian group, the group operation  $(\bs f, \bs g)\mapsto \bs f \bs g$ being given by component-wise multiplication of the corresponding tuples of rational functions. 
\end{defn}
In what follows, we do not always distinguish between a rational $\ell$-weight $\bs f$ and the corresponding tuple $(f_i(u))_{i\in I}$ of rational functions.
Note that in terms of the latter, we have \be\nn\wt(\bs f) =  (f_i(0))_{i\in I}.\ee 
%In what follows, we shall identify rational $\ell$-weights $\bs f$ with their corresponding tuples of rational functions $(f_i(u))_{i\in I}$ freely and without further comment.

For every weight $\varrho$, let $V(\varrho)$ be the irreducible $\uqg$-module with highest weight $\varrho$. Recall that $V(\varrho)$ is unique up to isomorphism and is finite-dimensional if and only if $\varrho\in \overline{P^+}$; see  \cite{CPbook}, chapter 10. 

For every rational $\ell$-weight $\bs f$, let us write $\L(\bs f)$ for the irreducible $\uqgh$-module with highest $\ell$-weight $\bs f$. By definition, $\L(\bs f)$ is unique up to isomorphism. Moreover $\L(\bs f)$ and $\L(\bs f')$ are not isomorphic unless $\bs f = \bs f'$. 
Every highest $\ell$-weight $\uqgh$-module with highest $\ell$-weight $\bs f\in \mc R$ has $\L(\bs f)$ as a quotient. 

Recall \cite{CP94} that $\L(\bs f)$ is finite-dimensional if and only for each $i\in I$ the rational function $f_i(u)$ is of the form $f_i(u) = q_i^{\deg P_i} P_i(u q_i^{-2}) / P_i(u)$ for some polynomial $P_i(u)$ with constant coefficient 1, called a Drinfeld polynomial. Observe that this is a stronger condition than $\wt(\bs f) \in \overline{P^+}$.

We can now state the following theorem, which classifies the simple objects in $\hat\cat$.

\begin{thm}\label{thm1}
The map $\bs f\mapsto \L(\bs f)$ defines a bijection between $\mc R$ and the isomorphism classes of simple objects in $\hat\cat$. 
\end{thm}
\begin{proof}
Suppose $V\in\Ob\hat\cat$ is irreducible. Then it follows from part (\ref{fgen}) of Definition \ref{catdef} that $V$ contains a singular $\ell$-weight vector, say $v$. Since $V$ is irreducible, $V=\uqgh.v$, so $V$ is a highest $\ell$-weight representation. 
Thus it is enough to show that a highest $\ell$-weight irreducible representation $V$ is in $\hat\cat$ if and only if its highest $\ell$-weight $\bs f$ is rational.

%By twisting by an automorphism as in Proposition \ref{c1prop} if necessary, let us suppose that $(c^{1/2}-\id)\on V= 0$.
%Now let $\C_{\bs f}$ denote the 1-dimensional $\hat U^0.\hat U^+$-module defined by 
%\be \hat U^+\on\C_{\bs f}=0, \quad (c^{1/2}- \id)\on \C_{\bs f} =0,\quad\text{and}\quad (\phi_{i,\pm r}^\pm - \gamma_{i,\pm r}^\pm \id)\on \C_{\bs f} = 0, \text{ for all } i\in I ,r\in \Z_{\geq 0}.\nn\ee 
%Let $M(\bs f) := \uqgh \otimes_{\hat U^0 . \hat U^+} \C_{\bs f}$ denote the induced $\uqgh$-module. It is a highest $\ell$-weight representation with highest $\ell$-weight $\bs f$, and its unique irreducible quotient is isomorphic to $V$. 

We shall first show that for each $i\in I$, $\dim(V_{\wt(\bs f)\overline{\alpha_i}^{-1}})<\8$ if and only if $\bs f$ is rational.
By (\ref{uqghtriang}), $V_{\wt(\bs f)\overline{\alpha_i}^{-1}}$ is spanned by the vectors $x_{i,r}^-.v$, $r\in \Z$. 
%It has dimension $<N$ if and only for all $k>N$ and for all $n_1,n_2,\dots,n_k\in \Z$, the vectors $x^-_{i,n_1}\on v, x^-_{i,n_2}\on v,\dots, x^-_{i,n_k}\on v\in V$ are linearly related; that is, if and only if there exist $c_1,\dots,c_k\in \C$ such that $\sum_{j=1}^k c_j x^-_{i,n_j} \on v$ is zero in $V$, and is therefore a singular vector in $M(\bs f)$.

Suppose $\dim(V_{\wt(\bs f)\overline{\alpha_i}^{-1}})<\8$. Then there exists $N>0$ and $a_1,a_2,\dots,a_N\in \C$, $a_1\neq 0 \neq a_N$, such that $\sum_{j=1}^N a_j x^-_{i,j} \on v = 0$. For all $s\in \Z$,
\be 0= (q_i-q_i^{-1}) x^+_{i,s}\sum_{j=1}^N a_j x^-_{i,j} \on v 
     = \sum_{j=1}^N a_j( f^+_{i,j+s} - f^-_{i,j+s}).  \label{singv}\ee
Here it is understood that $f^+_{-n} =f^-_{n}=0$ identically for all $n>0$.
So $0 = \sum_{j=1}^N a_j f^+_{i,j+s}$ for all $s\geq 0$. Letting $f_i^+(z) := \sum_{n=0}^\8 z^n f^+_n$, we have that 
\begin{align} f_i^+(z) \sum_{j=1}^N a_j z^{N-j} &= \sum_{j=1}^N \sum_{s=-j}^\8 z^{s+N} a_j f^+_{i,j+s} = \sum_{j=1}^N  \sum_{s=-j}^{-1}z^{s+N} a_j f^+_{i,j+s} \nn\\
&= \sum_{j=1}^N b_j z^{N-j}, \qquad\text{where}\qquad b_j = \sum_{\ell=1}^N a_\ell f^+_{i,\ell-j}. \nn\end{align}
Similarly, $0 = \sum_{j=1}^N a_j f^-_{i,j+s}$ for all $s<-N$. Letting $f_i^-(z) := \sum_{n=0}^\8 z^{-n} f_{-n}^-$, we have
\be f^-_i(z) \sum_{j=1}^N a_j z^{N-j} = \sum_{j=1}^N b'_jz^{N-j}, \qquad\text{where}\qquad b'_j = \sum_{\ell=1}^N a_\ell f^-_{i,\ell-j}. \nn\ee
The remaining equations of (\ref{singv}) are then $b_{-s}=b'_{-s}$ for $-N\leq s<0$.
Thus $f^+_i(z)$ and $f^-_i(z)$ are the Laurent expansions, about 0 and $\8$ respectively, of the rational function  
\be  \frac{ b_N + z b_{N-1} + z^2 b_{N-2} + \dots + z^{N-1} b_1}
                    { a_N + z a_{N-1} + z^2 a_{N-2} + \dots + z^{N-1} a_1}.\nn\ee  
Finally, the constraint that $f_{i,0}^+ f_{i,0}^-=1$ yields
\be \frac{b_N}{a_N} \frac{b_1}{a_1} = 1.\nn\ee
Hence $f_i^+(z)$ and $f_i^-(z)$ are indeed of the required form.

Conversely, suppose $f^+_i(z)$ and $f^-_i(z)$ are as above for some $N>0$. Then a similar calculation shows that $x^+_{i,s}\sum_{j=1}^N a_j x^-_{i,M+j} \on v=0$ for all $M\in \Z$. Since $V$ has no singular vectors which are not scalar multiples of $v$, it follows that $\sum_{j=1}^N a_j x^-_{i,M+j} \on v=0$, i.e. that for all $M\in \Z$, the vectors $\{x^-_{i,M+j}\on v:j=1,2,\dots,N\}$ are linearly related. By applying this result finitely many times, any given vector $x^-_{i,r}\on v$ can be expressed as a linear combination of, say, the vectors $\{x^-_{i,j}\on v:j=1,2,\dots,N\}$. That is, $V_{\wt(\bs f)\overline{\alpha_i}^{-1}}$ is finite-dimensional. 

To complete the proof, we note that if $V_{\wt(\bs f)\overline{\alpha_i}^{-1}}$ is finite-dimensional for every $i\in I$ then every remaining weight space $V_{\wt(\bs f)\overline{\alpha}^{-1}}$, $\alpha\in Q^+$,  of $V$ is finite-dimensional too. This follows by an induction on $\hgt(\alpha)$ exactly as in \cite{CP94}, \S5, proof of case (b). 
%Indeed, assume that $h\geq 2$ and suppose that $V_\varrho$ is finite-dimensional for all $\varrho$ such that $\hgt(\wt(\bs f)-\varrho)< h$. 
\end{proof}
The ``only if'' part of the theorem was proved in \cite{HernandezFusion}, Lemma 14.
%, so it can happen that the ``top'' $\uqg$-irreducible component $V(\wt(\bs f))$ of $\L(\bs f)$ is has finite dimension even when $\L(\bs f)$ does not.

\subsection{$q$-Characters}
Recalling the definition of the group $\mc R$ of rational $\ell$-weights, Definition \ref{rdef}, let us define a subgroup $\mc Q\subset \mc R$, the group of \emph{$l$-roots}, as follows.
For each $j\in I$ and $a\in \Cx$, define $\A_{j,a} \in \mc R$ by 
\be (\A_{j,a})_{i}(u) = q^{B_{ji}} \frac{1-q^{-B_{ji}} au}{1-q^{B_{ji}}au}\nn\ee
for each $i\in I$. Note that $\wt(\A_{j,a}) = \overline\al_j$. 
We call each $\A_{j,a}$ a \emph{simple $l$-root}.
The reader should be warned that in \cite{FR,FM} what we call $\A_{j,a}$ was instead labelled $\A_{j,aq_j}$.

Let $\Q$ be the subgroup of $\mc R$ generated by $\A_{i,a}$, $i\in I,a\in\Cx$. 
Note that $\mc Q$ is a free group, i.e. the $\A_{j,a}$ are algebraically independent. 
Let $\Q^\pm$ be the monoid generated by $\A_{i,a}^{\pm 1}$, $i\in I,a\in\Cx$. We call the latter the positive/negative $l$-roots.

There is a partial order $\leq$ on $\mc R$ in which $\bs f\leq \bs g$ if and only if $\bs g \bs f^{-1} \in \Q^+$. It is compatible with the partial order (\ref{pop}) on $\overline{\h^*}$ in the sense that $\bs f\leq \bs g$ implies $\wt \bs f\leq \wt \bs g$.

\begin{defn}
The \emph{$q$-character} of $V\in\Ob\hat\cat$ is the formal sum of its $\ell$-weights, counted with multiplicities:
\be \chi_q(V) := \sum_{\bs f\in \mc R} \dim(V_{\bs f}) \bs f \in \Z[\mc R].\nn\ee
\end{defn}
One also has the usual $\uqg$-character map
%\begin{defn}
%The \emph{character} of $V\in \Ob\cat$ is
\be \chi(V) := \sum_{\lambda \in \h^*} \dim(V_{\overline\lambda}) \overline\lambda \in \Z\left[\overline{\h^*}\right]\nn\ee
for any $V\in\Ob\cat$.
%\end{defn}
It is clear that $\chi(V) = (\wt\circ\chi_q)(V)$ for all $V\in \Ob\hat\cat$. 
\begin{prop}\label{stepprop}
Suppose $\bs f$ and $\bs g$ are $\ell$-weights of $V\in \Ob \hat \cat$, and $i\in I$. Then 
\be \nn V_{\bs g} \cap \bigoplus_{r\in \Z} x_{i,r}^\pm(V_{\bs f}) \neq \{0\} \Longrightarrow 
\bs g = \bs f \bs A^{\pm 1}_{i,a} \text{ for some } a\in \Cx.\ee 
\end{prop}
\begin{proof}
Let $(v_k)_{1\leq k\leq \dim V_{\bs f}}$ be a basis of $V_{\bs f}$ in which the action of the $\phi_{i,r}^\pm$ is upper-triangular, in the sense that for all $i\in I$ and $1\leq k\leq \dim V_{\bs f}$, 
\be(\phi_i^\pm(u)- f_{i}^\pm(u))\on v_k = \sum_{k'<k} v_{k'} \xi^{\pm,k,k'}_i(u),\label{uppertriang}\ee
for certain formal series $\xi^{\pm,k,k'}_i(u) \in u\C[[u]]$. (The leading order is $u^1$: recall that $\phi_{i,0}^\pm$ act diagonally.) Let $(w_k)_{1\leq k\leq \dim V_{\bs g}}$ be a basis of $V_{\bs g}$ in which the action of the $\phi_{i,r}^\pm$ is lower-triangular, in the sense that for all $i\in I$ and $1\leq k\leq \dim V_{\bs g}$, 
\be\nn(\phi_i^\pm(u)-g_i^\pm(u))\on w_{\ell} = \sum_{\ell'>\ell} w_{\ell'} \zeta_{i}^{\pm,\ell,\ell'}(u),\label{lowertriang}\ee
for certain formal series $\zeta^{\pm,\ell,\ell'}_i(u) \in u\C[[u]]$.
Consider for definiteness the case of $x^+_{j}(z)$ ($x^-_j(z)$ is similar).
%Let $K$ be smallest such that $(x_{i}^+(z) \on v_K)_{\bs g}\neq 0$. 
For all $1\leq k\leq \dim(V_{\bs f})$, 
\be (x_{i}^+(z) \on v_k)_{\bs g} = \sum_{\ell=1}^{\dim(V_{\bs g})} \lambda_{k,\ell}(z) w_\ell \nn\ee
for some formal series $\lambda_{k,\ell}(z)\in \C[[z]]$ for each $\ell$, $1\leq \ell \leq \dim(V_{\bs g})$.
It follows from the defining relations (\ref{hxpm}) that
\be (q^{B_{ij}}-uz) x^+_i(z) \left(\phi^+_j(u) - f^+_j(u) \right) \on v_k = \left( (1-q^{B_{ij}}uz)\phi^+_j(u) -  (q^{B_{ij}}-uz)  f^+_j(u) \right) x^+_i(z)\on v_k\nn\ee
where $x^+_i(z) := \sum_{r\in \Z} z^{-r} x_{i,r}^+$. On resolving this equation in the basis of $V_{\bs g}$ above and taking the $w_{\ell}$ component, we have
\begin{align}(q^{B_{ij}}-uz)\sum_{k'=1}^{k-1} \xi^{+,k,k'}_j(u) \lambda_{k',\ell}(z) &= \left( (1-q^{B_{ij}}uz)g^+_j(u) -  (q^{B_{ij}}-uz) f^+_j(u) \right) \lambda_{k,\ell}(z) \nn\\
&\phantom = + (1-q^{B_{ij}}uz) \sum_{\ell'=1}^{\ell-1} \lambda_{k,\ell'}(z) \zeta_j^{+,\ell',\ell}(u) 
\label{XAeqn}.\end{align}
Suppose $V_{\bs g} \cap (x_{i}^+(z)(V_{\bs f})) \neq \{0\}$. Then there is a smallest $K$ such that $(x_{i}^+(z) \on v_K)_{\bs g} \neq 0$ and then a smallest $L$ such that $\lambda_{K,L}(z)\neq 0$. So (\ref{XAeqn}) gives, in particular,
\be 0= \left( (1-q^{B_{ij}}uz)g^+_j(u) -  (q^{B_{ij}}-uz)f^+_j(u) \right) \lambda_{K,L}(z) 
\label{Xeqn}.\ee
This must hold for all $j\in I$. For each $j\in I$, (\ref{Xeqn}) is an equation of the form $0=\lambda_k(v) \sum_{n=0}^\8 u^n (b^{(i)}_n + c^{(i)}_nv)$  for the formal Laurent series $\lambda_{K,L}(v)$,  with $b^{(i)}_n,c^{(i)}_n\in \C$ for all $n\in\Z_{\geq 0}$. Equivalently, for each $i\in I$, it is a countably infinite set of first order recurrence relations on the series coefficients of $\lambda_{K,L}(v)$. There are non-zero solutions if and only if there is an $a\in\Cx$ such that $b^{(i)}_n/c^{(i)}_n=-a$ for all $n\in \Z_{\geq 0}$ and all $j\in I$. That is,
\be g^+_j(u) \left(f^+_j(u)\right)^{-1} = q^{B_{ij}} \frac{1-q^{-B_{ij}} u a}{1-q^{B_{ij}}u a} \nn\ee
as an equality of power series in $u$. Similar arguments hold for $\phi_i^-(u)$.  
\end{proof}
This proposition has a number of important corollaries. First,
\begin{cor}\label{taucor}
Suppose $\bs f$ and $\bs g$ are $\ell$-weights of $V\in \Ob \hat \cat$, and $v\in V_{\bs f}$ and $w\in V_{\bs g}$ are nonzero. 
If $i\in I$ and $a\in \Cx$ are such that $w\in\Span_{r\in \Z} x_{i,r}^{\pm} \on v$ and  
\be (\bs g)_i = (\bs f \bs A_{i,a}^{\pm 1})_i,\qquad\text{i.e.}\qquad g_i(u) = f_i(u) q_i^2 \frac{1-q_i^{-2}au}{1-q_i^2au}, \nn\ee
then $\bs g = \bs f \bs A_{i,a}^{\pm 1}$.    \qed
\end{cor}
As is the case for finite-dimensional $\uqgh$-modules, the $q$-characters of the simple objects in  $\hat\cat$ have the following ``cone'' property.
\begin{cor}\label{coneprop}
For all $\bs f\in \mc R$,  $\chi_q(\L(\bs f)) \in \bs f \Z[\A_{i,a}^{-1}]_{i \in I, a\in \Cx}$. In particular, all $\ell$-weights of $\L(\bs f)$ are rational.
\end{cor}
\begin{proof}
Given (\ref{uqghtriang}) and Proposition \ref{stepprop}, this follows from Theorem \ref{thm1}.
\end{proof}

Let $\groth{\hat\cat}$ be the Grothendieck ring of $\hat\cat$. For all $V\in \Ob\hat\cat$, the class $[V]\in \groth{\hat\cat}$ is a $\Z$-linear combination of the classes $[\L(\bs f)]\in \groth{\hat\cat}$, $\bs f\in \mc R$, of the irreducibles in $\hat\cat$. 
\begin{thm}\label{chiqhom}
$\chi_q$ defines an injective ring homomorphism  $\groth{\hat\cat}\to \Z[\mc R]$.
\end{thm}
\begin{proof}
It is clear that $\chi_q(W) = \chi_q(U) + \chi_q(V)$ whenever $U,V,W\in \Ob\hat\cat$ are such that $[W] = [U] + [V]$ in $\groth{\hat\cat}$, i.e. whenever there is a short exact sequence $0\to V\to W \to U \to 0$ of $\uqgh$-modules. 

To show that $\chi_q(V\otimes W) = \chi_q(V)\chi_q(W)$ we argue as in \cite{FR}. For each $\ell$-weight $\bs f$ of $V$, let $(v_{\bs f,k})_{1\leq k\leq \dim V_{\bs f}}$ be a basis of $V_{\bs f}$ in which the action of the $\phi_{i,r}^\pm$ is upper-triangular, c.f. (\ref{uppertriang}); and likewise for each $\ell$-weight $\bs g$ of $W$ let $(w_{\bs g,k})_{1\leq k\leq \dim W_{\bs g}}$ be an upper-triangular basis of $W_{\bs g}$. Then it follows from (\ref{damiani}) that $(v_{\bs f,k}\otimes w_{\bs g,\ell})$ is a basis of $(V\otimes W)_{\bs f \bs g}$ in which the action of the $\phi_{i,r}^\pm$ is upper-triangular. Thus, $\ell$-weights are multiplicative across tensor products, and their multiplicities are additive, as required. 

The classes $[\L(\bs f)]\in \groth{\hat\cat}$, $\bs f\in \mc R$, of the irreducible representations are linearly independent, because their images under $\chi_q$ are linearly independent. Injectivity of $\chi_q$ follows from this.
\end{proof}

\begin{cor}\label{allratprop}
All $\ell$-weights of representations in $\hat\cat$ are rational.\qed
%Suppose $V\in\Ob\hat\cat$ and $\bs f$ is an $\ell$-weight of $V$, i.e. $V_{\bs f}\neq\{0\}$. Then $\bs f\in \mc R$.  \qed
%Then there is an $N$-tuple of complex-valued rational functions $(f_i(u))_{i\in I}$ of a formal variable $u$ such that, for each $i\in I$, $f_i(u)$ is regular at $0$ and $\8$,  $f_i(0)f_i(\8)=1$, and
%\be \sum_{r=0}^\8 f^+_{i,r} u^r = f_i(u) = \sum_{r=0}^\8 f^-_{i,-r} u^{-r} \label{rform2}\ee
%in the sense that the left- and right-hand sides are the Laurent expansions of $f_i(u)$ about $0$ and $\8$, respectively.
\end{cor}

We also need the following proposition.

\begin{prop} \label{irrbothprop}
Suppose $\L(\bs f)\in\Ob\hat\cat$ and $\L(\bs g)\in\Ob\hat\cat$ are such that $\L(\bs f)\otimes \L(\bs g)$ is irreducible. Then $\L(\bs f) \otimes \L(\bs g) \cong \L( \bs f  \bs g) \cong \L(\bs g) \otimes \L(\bs f)$ as $\uqgh$-modules. 
\end{prop}
\begin{proof}
Let $v\in \L(\bs f)$ and $w\in \L(\bs g)$ be highest $\ell$-weight vectors. 
Since $\L(\bs f)\otimes \L(\bs g)$ is irreducible, to show that it is isomorphic to $\L(\bs{fg})$ it is, by Theorem \ref{thm1}, enough to show that $v\otimes w$ is a singular $\ell$-weight vector in $\L(\bs f) \otimes \L(\bs g)$ and has $\ell$-weight $\bs f  \bs g$. 
That $v\otimes w$ has $\ell$-weight $\bs f \bs g$ follows from (\ref{damiani}). 
That $v\otimes w$ is singular follows exactly as in the case of finite-dimensional modules, c.f. \cite{CPbook}.
Finally, $\L(\bs g) \otimes \L(\bs f)$ contains $\L(\bs f\bs g)=\L(\bs f)\otimes \L(\bs g)$ as a subquotient. But $\chi(\L(\bs f)\otimes\L(\bs g))=\chi(\L(\bs g)\otimes \L(\bs f))$. Hence $\L(\bs g) \otimes \L(\bs f)\cong \L(\bs f \bs g)$. 
\end{proof}
\subsection{Analytic continuation}\label{acsec}
In this subsection we observe that if the rational highest $\ell$-weight $\bs f$ depends rationally on an additional parameter $x\in \C$ then the normalized $\uqg$-character of $\L(\bs f)$,
\be \widetilde\chi(\L(\bs f)) := \overline{\lambda}^{-1}\chi(\L(\bs f)),\quad \text{where  $\overline\lambda= \wt(\bs f)$},\nn\ee 
is the same for almost all $x$. 
In fact, for each positive integer $n$, $\widetilde\chi(\L(\bs f))$ modulo weights $\mu$ such that $\hgt(\lambda-\mu)>n$,  is the same for all but \emph{finitely many} $x$.

We use the following standard lemma from linear algebra. 
\begin{lem}\label{dimdroplem}
Let $V,W$ be complex vector spaces, with $\dim W<\8$, and let  $A_i(u): V\to W$, $i\in \N$, be a countable set of linear operators rationally depending on a complex parameter $x$. Let $d_{A(x)} = \codim_V\left(\bigcap_{i\in \N} \ker A_i(x)\right) $. 
Assume that for all $x\in \C$, $d_{A(x)}<\8$. 

Then there exists a finite set $S\subset \C$ such that for all $x_1,x_2\in \C\setminus S$ and $x_3\in S$ we have $d_{A(x_1)} = d_{A(x_2)} \geq d_{A(x_3)}$.\qed
\end{lem}
\begin{prop}\label{anacprop}
Let $f_i(u,x)$, $i\in I$, be rational functions of $u$ and $x$ such that for each $x\in \C$, $f_i(u,x)$ defines a rational $\ell$-weight $\bs f(x)$. Then for all $\alpha\in \Q^+$ there exists a finite set $S\subset \C$ such that for all $x_1,x_2\in \C\setminus S$ and $x_3\in S$ we have
\be\nn \dim\left(\L\left(\bs f\left(x_1\right)\right)\right)_{\wt(\bs f)\overline\alpha^{-1}} = \dim\left(\L\left(\bs f\left(x_2\right)\right)\right)_{\wt(\bs f)\overline\alpha^{-1}}
\geq \dim\left(\L\left(\bs f\left(x_3\right)\right)\right)_{\wt(\bs f)\overline\alpha^{-1}}.\ee
\end{prop}
\begin{proof}
By induction on $\hgt(\alpha)$, making use of Lemma \ref{dimdroplem}. 
\end{proof}

\subsection{Dual modules}\label{dualsec}
Given $V\in \Ob\cat$ we shall write $V^*$ for the restricted left dual of $V$. That is, $V^*$ is the space of linear maps $\lambda: V\to \C$ with finite support on a weighted basis of $V$, equipped with the left $\uqg$-action given by $(x\on \lambda)(v) = \lambda(S(x)\on v)$. 
It is clear that $V^*$ is a weight module whose weight spaces are all finite-dimensional.

If $V\in \Ob\cat$ is highest weight then $V^*$ is lowest weight. 

If $V\in \Ob\hat\cat$ then $V^*$ is also a $\uqgh$-module; moreover if $V$ is highest $\ell$-weight then $V^*$ is lowest $\ell$-weight, in the obvious sense. 

Let $R(\bs g)$ denote the irreducible lowest $\ell$-weight $\uqgh$-module with lowest $\ell$-weight $\bs g$. 
\begin{prop}
For all $\bs f \in \mc R$, $\L(\bs f)^*\cong R(\bs f^{-1})$ as $\uqgh$-modules.
\end{prop}
\begin{proof}
$\L(\bs f)^*$ is irreducible and so isomorphic to some $R(\bs g)$; we shall now show that $\bs g=\bs f^{-1}$.  Indeed, by definition the following diagram commutes for all $x\in \uqgh$:
\be\label{tauj}\begin{tikzpicture}    
\matrix (m) [matrix of math nodes, row sep=3em,    
column sep=4em, text height=2ex, text depth=1ex]    
{     
\L(\bs f)^*\otimes \L(\bs f) & \L(\bs f)^*\otimes \L(\bs f) \\    
\C & \C\\    
};    
\path[->,font=\scriptsize,shorten <= 2mm,shorten >= 2mm]    
(m-1-1) edge node [above] {$x$} (m-1-2)    
(m-2-1) edge node [above] {$\id$} (m-2-2)    
(m-1-1) edge (m-2-1)    
(m-1-2) edge (m-2-2);    
\end{tikzpicture} \quad
\begin{tikzpicture}    
\matrix (m) [matrix of math nodes, row sep=3em,    
column sep=4em, text height=2ex, text depth=1ex]    
{     
\lambda\otimes v & x\on(\lambda\otimes v) =  x^{(1)}\on\lambda\otimes x^{(2)}\on v  \\    
\lambda(v) & \lambda(S(x^{(1)}) x^{(2)}\on v) = \lambda(v).\\    
};    
\path[|->,font=\scriptsize,shorten <= 2mm,shorten >= 2mm]    
(m-1-1) edge (m-1-2)    
(m-2-1) edge (m-2-2)    
(m-1-1) edge (m-2-1)    
(m-1-2) edge (m-2-2);    
\end{tikzpicture}\nn\ee    
Now suppose we take $\lambda$ to be the lowest weight vector in $\L(\bs f)^*$ and $v$ to be the highest weight vector in $\L(\bs f)$. Note $\lambda(v)\neq 0$. It follows from (\ref{damiani}) that
\be \phi_i^\pm(u)\on(\lambda\otimes v) =  \phi_i^\pm(u)\on\lambda\otimes \phi_i^\pm(u)\on v   = g^\pm_i(u) f^\pm_i(u)  \lambda \otimes v.\nn\ee
Therefore $g^\pm_i(u) f^\pm_i(u)  \lambda(v) =\lambda(v)$ identically, which can hold only if the rational functions $f_i(u)$ and $g_i(u)$ obey $g_i(u) f_i(u) = 1$, as claimed.
\end{proof} 
Given a rational $\ell$-weight $\bs f$, let us define $\bs f^\dag$ by
\be f_i^\dag(u) = \frac 1 {f_i(u^{-1})}.\label{dagdef}\ee
Note that $\bs f^\dag$ is again a rational $\ell$-weight, and that $(\bs f^\dag)^\dag = \bs f$.
From (\ref{cinvhatdef}--\ref{cinvhatphi}) one sees that $R(\bs f^{-1})^{\hat\cinv}\cong L(\bs f^\dag)$, where ${}^{\hat\cinv}$ denotes the pull-back via the Cartan involution. Hence we have the following.
\begin{cor}\label{flipcor}
For all $\bs f\in \mc R$, $\L(\bs f^\dag)\cong(\L(\bs f)^*)^{\hat\cinv}$ as $\uqgh$-modules.\qed
\end{cor}

\section{Description of irreducibles in category $\hat\cat$ when $\g=\mf{sl}_2$}\label{sl2sec}
Throughout this section, $\g=\mf{sl}_2$. Recall \cite{Jimbo,CPsl2} that for any $a\in \Cx$ there is a homomorphism of algebras $\ev_a:\uqslth\to\uqslt$ such that $\ev_a(c^{1/2})=1$ and
\be \ev_a(x_{1,r}^+) = q^{-r} a^{r} k_1^r x_1^+,\quad 
    \ev_a(x_{1,r}^-) = q^{-r} a^{r} x_1^- k_1^r.\label{evmap}\ee
These maps are called \emph{evaluation homomorphisms}, and the pull-backs of $\uqslt$-modules by the $\ev_a$ are called  \emph{evaluation modules}. One of the first key results in the theory of finite-dimensional representations of quantum affine algebras is that every irreducible $\uqslth$-module is isomorphic to a tensor product of evaluation modules \cite{CPsl2}. In this section we give the analogous description of the irreducibles in $\hat\cat$. 
\subsection{Strings}
Let $V(\mu)_a\in\Ob\hat\cat$ denote the pull-back via $\ev_a$ of the irreducible $\uqslth$-module $V(\overline{\mu\omega_1})$, $a\in \Cx$, $\mu\in \C$. It is finite-dimensional if and only if $\mu\in \Z_{\geq 0}$. It is irreducible, with highest $\ell$-weight given by the rational function
\be S_\mu(a) :u\mapsto  q^\mu\frac{1-q^{-\mu-1} au}{1-q^{\mu-1} au}.\label{Sdef}\ee
We refer to any rational function of $u$ of this form as a \emph{string}.
\begin{defn}
We say that two strings $S_\mu(a)$ and $S_\nu(b)$, $a,b\in \Cx$,  $\mu,\nu\in \C$, are in \emph{general position} if 
\begin{enumerate}
\item[(1)] if $\mu\notin \Z_{\geq 0}$ and $\nu\notin \Z_{\geq 0}$ then $aq^{-\mu-1}\notin bq^{\nu-1-2\Z_{\geq 0}}$ and $bq^{-\nu-1}\notin aq^{\mu-1-2\Z_{\geq 0}}$;
\item[(2i)] if $\mu\in \Z_{\geq 0}$ and $\nu\notin\Z_{\geq 0}$ then neither $bq^{-\nu-1}$ nor $bq^{\nu-1}$ lie in $aq^{\mu-1-2\Z_{\geq 0}}\cap aq^{-\mu-1+2\Z_{\geq 0}}$;
\item[(2ii)] if $\nu\in \Z_{\geq 0}$ and $\mu\notin\Z_{\geq 0}$ then neither $aq^{-\mu-1}$ nor $aq^{\mu-1}$ lie in $bq^{\nu-1-2\Z_{\geq 0}}\cap bq^{-\nu-1+2\Z_{\geq 0}}$;
\item[(3)] if $\mu\in\Z_{\geq 0}$ and $\nu\in\Z_{\geq 0}$ then the sets \be aq^{\mu-1-2\Z_{\geq 0}}\cap aq^{-\mu-1+2\Z_{\geq 0}}\quad\text{and}\quad bq^{\nu-1-2\Z_{\geq 0}}\cap bq^{-\nu-1+2\Z_{\geq 0}}\nn\ee are either disjoint, or one is contained in the other.
\end{enumerate} 
\end{defn}
We say that the string $S_\mu(a)$ \emph{starts} at $aq^{-\mu-1}$ and \emph{ends} at $aq^{\mu-1}$. We call a string \emph{finite} if it starts to the left of its end, where we say $a$ is \emph{to the left} 
%(resp. \emph{to the right}) 
of $b$ if $a\in bq^{-2\Z_{\geq 0}}$. 
%(resp. $a\in bq^{2\Z_{\geq 0}}$) 
Thus $S_\mu(a)$ is finite if and only if $\mu\in\Z_{\geq 0}$. If $S_\mu(a)$ is finite we associate it with the finite set 
\be aq^{\mu-1-2\Z_{\geq 0}}\cap aq^{-\mu-1+2\Z_{\geq 0}} = \{aq^{-\mu-1},aq^{-\mu+1},\dots,aq^{\mu-3}, aq^{\mu-1}\},\nn\ee and we say $b$ is \emph{inside} $S_\mu(a)$ if it belongs to this set. 
In this language, two strings are in general position if and only if 
\begin{enumerate}
\item if neither string is finite then neither string starts to the left of the end of the other;
\item if one string is finite and the other is not, then the non-finite string neither starts nor ends inside the finite one;
%\item if $S_\nu(b)$ is finite and $S_\mu(a)$ is not, then $S_\mu(a)$ neither starts nor ends inside $S_\nu(b)$;
\item if both strings are finite then their sets are either disjoint or one is contained in the other. 
\end{enumerate}
The final part is the usual condition for finite-dimensional representations, c.f. \cite{CPsl2,CP94}. The reader should be warned that in \cite{CPsl2} the $q$-string corresponding to (\ref{Sdef}) is defined not to include the element $aq^{-\mu-1}$, in contrast to our convention.

\begin{prop}
If $\L(\bs f)\in\Ob\hat\cat$ then the corresponding rational function $f(u)$ can be written in the form
\be f(u)= \prod_{k=1}^r S_{\mu_k}(a_k), \quad \mu_1,\dots,\mu_r\in \Cx,\,\, a_1,\dots,a_r\in \Cx \label{facf}\ee 
in such a way that each pair $(S_{\mu_i}(a_i), S_{\mu_j}(a_j))$, $1\leq i< j\leq r$, is in general position. 
\end{prop}
\begin{proof}
It is clear that any rational function $f(u)$ obeying the conditions of Definition \ref{rdef} can be written in the form (\ref{facf}). (We require $\mu_1,\dots,\mu_r\neq 0$ so that all factors are non-trivial. If $f(u)=1$ we have $r=0$.) To see that these factors may be chosen to be pairwise in general position we argue as follows. If not all pairs are in general position, then by definition we can always find some pair, call it $(S_{\mu}(a),S_\nu(b))$, such that
$bq^{-\nu-1}\in aq^{\mu-1-2\Z_{\geq 0}}$ but $bq^{\nu-1}\notin aq^{\mu-1-2\Z_{\geq 0}} \cap bq^{-\nu-1+2\Z_{\geq 0}}$. Let us write
\be S_\mu(a) =: \frac{A^{-1}-Au}{B^{-1}-Bu},\quad S_\nu(b) =: \frac{C^{-1}-Cu}{D^{-1}-Du}.\nn\ee
We swap this pair for the new pair $(S_{\mu'}(a'),S_{\nu'}(b'))$ defined by
\be S_{\mu'}(a') :=  \frac{A^{-1}-Au}{D^{-1}-Du},\quad S_{\nu'}(b') := \frac{C^{-1}-Cu}{B^{-1}-Bu}.\nn\ee
Obviously $S_{\mu}(a) S_\nu(b) = S_{\mu'}(a')S_{\nu'}(b')$. By inspection one checks that the new pair are in general position. We shall now argue that after some finite number of such swaps all pairs will be in general position. Consider the partial ordering on tuples $(\mu_1,\dots,\mu_r)\in \Cx$ defined as follows: let $s(\mu_1,\dots,\mu_r)$ be the weakly increasing $r$-tuple obtained by discarding any $\mu_k\notin \Z_{>0}$,  sorting those that remain into weakly increasing order, and then appending entries $\8$ as needed. Then we say $(\mu_1,\dots,\mu_r) < (\nu_1,\dots,\nu_r)$ if and only if $s(\mu_1,\dots,\mu_r)$ precedes $s(\nu_1,\dots,\nu_r)$ lexicographically; that is, if and only if for some $k\geq 1$, $s(\mu_1,\dots,\mu_r)_k < s(\nu_1,\dots,\nu_r)_k$ and $s(\mu_1,\dots,\mu_r)_\ell = s(\nu_1,\dots,\nu_r)_\ell$ for all $1\leq \ell < k$. 
In the swapping procedure above at least one of $\mu',\nu'$ is always a positive integer and moreover \be\nn\min(\{\mu',\nu'\}\cap \Z_{>0})<\min((\{\mu,\nu\}\cap \Z_{>0})\cup\{\8\}).\ee Thus, repeated swapping produces a strictly decreasing sequence of tuples. 
%From a given starting tuple, all such sequences are finite. 
So the swapping process must terminate, and all pairs are then in general position.
\end{proof}
\begin{rem}
In contrast to the case of finite-dimensional representations, this factorization is not always unique. For example 
\bea S_{-5}(a)S_{-9}(a) &=&\nn q^{-5}\frac{1-q^{4}au}{1-q^{-6}au}q^{-9}\frac{1-q^{8}au}{1-q^{-10}au} \\ &=&
q^{-7}\frac{1-q^{4}au}{1-q^{-10}au}q^{-7}\frac{1-q^{8}au}{1-q^{-6}au} = S_{-7}(q^{-2}a) S_{-7}(q^2a) \nn\eea
and both $(S_{-5}(a),S_{-9}(a))$ and $(S_{-7}(q^{-2}a), S_{-7}(q^2a))$ are in general position.
\end{rem}
\subsection{Irreducible tensor products}
The following proposition, stated in \cite{CPsl2} for finite-dimensional modules, remains valid in $\hat\cat$. 
\begin{prop}\label{sl2basis}
There is a basis $(v_i)_{0\leq i \leq \dim(V(\mu)_a)-1}$ of $V(\mu)_a$ on which the action of the generators $x^\pm_{1,k}$ is given by:
\begin{align} 
x_{1,k}^+ \on v_i &= a^k q^{k(\mu-2i+1)} [\mu-i+1]_q v_{i-1}, \nn\\
x_{1,k}^- \on v_i &= a^k q^{k(\mu-2i-1)} [i+1]_q v_{i+1} .\nn\end{align}
(Here $v_{-1}\equiv 0$ and, if $\dim(V(\mu)_a)<\8$, $v_{\dim(V(\mu)_a)}\equiv 0$.)
\qed\end{prop}
\begin{proof}
The check is straightforward, using (\ref{evmap}) and the usual basis of the irreducible $\uqslt$-module $V(\overline{\mu\omega_1})$.
\end{proof}
\begin{prop}\label{sl2dualbasis}
In the dual basis $(v^*_i)_{0\leq i \leq \dim(V(\mu)_a)-1}$ of $(V(\mu)_a)^*$ the action of the generators $x^\pm_{1,k}$ is given by:
\begin{align} 
x_{1,k}^+ \on v^*_i &= -a^{k} q^{-k(\mu-2i-3)-(\mu-2i-2)}[\mu-i]_q v^*_{i+1}, \nn\\
x_{1,k}^- \on v^*_i &=-a^{k} q^{-k(\mu-2i-1) + (\mu-2i)} [i]_q v^*_{i-1} .\nn\end{align}
\end{prop}
\begin{proof} By direct calculation, making use of the relation $\ev_a\circ S = S\circ \ev_{aq^2}$ satisfied by the antipode of $\uqslth$ \cite{CPsl2}. (One checks this equality on the Chevalley generators (\ref{chevgens}), using the relations $k_0\equiv k_1^{-1}$, $x_{0}^+\equiv x_{1,1}^-k_1^{-1}$ and $x_0^- \equiv k_1 x_{1,-1}^{+}$.)\end{proof}
\begin{thm}\label{sl2thm} Let $a_1,\dots,a_r\in \Cx$ and $\mu_1,\dots,\mu_r\in \Cx$. The tensor product 
\be\label{nap} V(\mu_1)_{a_1} \otimes \dots \otimes V(\mu_r)_{a_r} \ee 
is irreducible if and only if each pair $(S_{\mu_i}(a_i), S_{\mu_j}(a_j))$, $1\leq i< j\leq r$, is in general position. 
\end{thm}
\begin{proof}
We show this first for the case $r=2$. Consider $V(\mu)_a\otimes V(\nu)_b$. As representations of $\uqslt$,
\be V(\overline{\mu\omega_1})\otimes V(\overline{\nu\omega_1}) = \bigoplus_{i=0}^M V\left(\overline{(\mu+\nu-2i)\omega_1}\right)\nn \ee
where 
\be M= \begin{cases} \min(\{\mu,\nu\} \cap \Z_{\geq 0}) & \text{if }\{\mu,\nu\} \cap \Z_{\geq 0}\neq \emptyset, \\
                          \8  & \text{otherwise}.\end{cases}\nn\ee
For each $0\leq p\leq M$, let $\Omega_p\in V\left(\overline{\mu\omega_1}\right)\otimes V(\overline{\nu\omega_1})$ be a $\uqslt$-highest weight vector of the $\uqslt$-submodule $V\left(\overline{(\mu+\nu-2p)\omega_1}\right)$. Exactly as in \cite{CPsl2}, \S4.8, one verifies that, for each $1\leq p\leq M$, $\Omega_p$ generates a proper $\uqslth$-submodule of $V(\mu)_a\otimes V(\nu)_b$ if and only if
\be \frac{b}{a} = q^{\mu+\nu-2p+2}.\nn\ee
Thus $V(\mu)_a\otimes V(\nu)_b$ has a submodule not containing $\Omega_0$ if and only if
\be \frac{b}{a} \notin \{q^{\mu+\nu-2p +2} : 1\leq p\leq M\}.\nn\ee

$(V(\mu)_a\otimes V(\nu)_b)^*\cong (V(\nu)_b)^* \otimes (V(\mu)_a)^*$ is lowest weight as a $\uqslt$-module. 
Let $\Omega_0^*$ be a $\uqslt$-lowest weight vector of $(V(\mu)_a\otimes V(\nu)_b)^*$. It is unique up to scale and it annihilates every element except $\Omega_0$ of any basis of $V(\mu)_a\otimes V(\nu)_b$ consisting of weight vectors.

The module $V(\mu)_a\otimes V(\nu)_b$ has a proper submodule containing $\Omega_0$ if and only if $(V(\nu)_b)^* \otimes (V(\mu)_a)^*$ has a factor module containing $\Omega_0^*$; that is, if and only if $(V(\nu)_b)^* \otimes (V(\mu)_a)^*$ has a submodule not containing $\Omega_0^*$. Given Proposition \ref{sl2dualbasis}, a similar calculation to the one referred to above shows that this is the case if and only if 
\be \frac{a}{b} \notin \{q^{\mu+\nu-2p +2} : 1\leq p\leq M\}.\nn\ee

Therefore, $V(\nu)_a\otimes V(\mu)_b$ has \emph{no} proper submodule, i.e. is irreducible, if and only if
\be \frac{a}{b} \notin \{q^{\pm(\mu+\nu-2p +2)} : 1\leq p\leq M\}.\nn\ee
By inspection one verifies that this condition holds precisely when $(S_\mu(a),S_\nu(b))$ are in general position. 

Turning to the general case, for the ``only if'' part we argue as follows. Suppose some pair $(S_{\mu_i}(a_i), S_{\mu_j}(a_j))$, $1\leq i< j\leq r$, is not in general position. If the tensor product (\ref{nap}) is irreducible then it is irreducible for all orderings of the tensor factors, c.f. Proposition \ref{irrbothprop}. So it is enough to show it is reducible for some ordering of the tensor factors. Pick any ordering in which $V(\mu_i)_{a_i}$ and $V(\mu_j)_{a_j}$ are adjacent; then the tensor product is indeed reducible, because it has a factor $V(\mu_i)_{a_i}\otimes V(\mu_j)_{a_j}$ which is reducible, as above. 

Now we prove the ``if'' part. The argument is essentially as in \cite{CPsl2}, and is by an induction on the number $r$ of tensor factors. We have the case $r=2$ above. For the inductive step, we may suppose that $V(\mu_1)_{a_1} \otimes\dots\otimes V(\mu_{r-1})_{a_{r-1}}$ is generated as a $\uqslth$-module by a tensor product of highest weight vectors of the tensor factors, 
\be\nn\Omega':=v_0\otimes \dots \otimes v_0\in V(\mu_1)_{a_1} \otimes\dots\otimes V(\mu_{r-1})_{a_{r-1}},\ee and therefore that $V(\mu_1)_{a_1} \otimes \dots\otimes V(\mu_{r-1})_{a_{r-1}}\otimes V(\mu_r)_{a_r}$ is generated as a $\uqslth$-module by the vectors $(\Omega'\otimes v_{i})_{0\leq i\leq \dim(V(\mu_r)_{a_r}}$.  We now argue by induction on $i$ that $\Omega'\otimes v_i\in \uqslth \on \Omega$, where $\Omega := \Omega'\otimes v_0$.  This is trivially true for $i=0$. For the inductive step, suppose $\Omega'\otimes v_j\in \uqslth \on \Omega$ for all $0\leq j\leq i$, and consider the action of $x_{1,k}^-$ on $\Omega'\otimes v_i$. 

Recall from \cite{CPsl2} the following property of the comultiplication of the quantum loop algebra $U:= \uqslth/(c^{1/2}-\id)$. Let $X_\pm$ be the subspaces of $U$ spanned by $(x^\pm_{1,k})_{k\in \Z}$. Then for all $k\in\Z_{\geq 1}$, 
\be\Delta x_{1,k}^- \equiv x_{1,k}^- \otimes 1 + \sum_{i=0}^{k-1} \phi^+_{1,i} \otimes x^-_{1,k-i} \quad\text{modulo}\quad UX_+ \otimes UX_-^2.\nn\ee
Therefore if we let $d_{k,j}$ be the eigenvalue of $\phi^+_{1,k}$ on the highest weight space of $V(\mu_1)_{a_1} \otimes\dots\otimes V(\mu_{j})_{a_{j}}$, set $b_r:=a_rq^{\mu_r-2i}$ and $b_j:=a_jq^{\mu_j}$ for each $1\leq j<r$, and define 
\be A_{k,j} := \sum_{p=0}^{k-1} d_{p,j} b_{j+1}^{k-p},\nn\ee
then for all $k\in \Z_{\geq 1}$ we have
\be x_{1,k}^- \on (\Omega'\otimes v_i) = \sum_{j=0}^{r-1} A_{k,j} \left(\id^{\otimes j} \otimes x^-_1 \otimes \id^{\otimes (r-j-1)}\right) \on (\Omega'\otimes v_i).\label{nxtvec} \ee
Note that $A=(A_{k,j})_{1\leq k\leq r,0\leq j\leq r-1}$ is a square matrix. To complete the proof it is enough to show that $\det A\neq 0$, for then equation (\ref{nxtvec}) allows $\Omega'\otimes v_{i+1}$ to be expressed as a linear combination of $x_{1,k}^-\on (\Omega'\otimes v_i)$, $1\leq k\leq r$, which completes the inductive step on $i$, and consequently also the inductive step on $r$. 

It was shown in \cite{CPsl2} that 
\be \det A = q^{\sum_{j=1}^{r-1} \mu_j} \prod_{j=1}^t b_j \prod_{k>j} \left(b_k-q^{-2\mu_j} b_j\right).\nn\ee
If $\det A= 0$ then there exist $j,k$ such that either
\begin{enumerate}
\item $1\leq j<k<r$ and $a_j=q^{\mu_j+\mu_k} a_k$, or else  
\item $1\leq j<k=r$ and $a_j=q^{\mu_j+\mu_r-2i}a_r$.
\end{enumerate}
The first of these is impossible since $S_{\mu_j}(a_j)$ and $S_{\mu_k}(a_k)$ are in general position. 
For the second, by making use of the freedom noted above to reorder the tensor factors, we may assume that $\dim V(\mu_r)_{a_r}\leq \dim(V(\mu_j)_{a_j})$ for all $1\leq j\leq r$. That is, if any of the tensor factors have finite dimension, then none have dimension lower than $V(\mu_r)_{a_r}$. 
Now if $\mu_r\in \Z_{\geq 0}$ then $i<\mu_r$ and so (2) is also ruled out since $S_{\mu_j}(a_j)$ and $S_{\mu_k}(a_k)$ are in general position.
\end{proof}

%Writing $\bs S_\mu(a)\in \mc R$ for the rational $\ell$-weight corresponding to the string $S_\mu(a)$, t
\begin{rem*}
The $q$-character of $V(\mu)_a \cong \L(S_\mu(a))$ is 
\be \chi_q( \L(S_\mu(a))) = S_\mu(a) \sum_{k=0}^{\dim \L(S_\mu(a))-1} \,\, \prod_{\ell=0}^{k-1} \bs A_{1,aq^{\mu-1+2\ell}}^{-1}.\nn\ee
\end{rem*}

\section{Least affinizations}\label{sec:minaff} 
In this section our main result is Theorem \ref{lathm}, which classifies the least affinizations in types ABCF and G. 
\subsection{Definition of least affinizations} It is natural to consider affinizations of the simple objects of $\cat$, in the sense of the following definition, which is adapted directly from the case of finite-dimensional representations discussed in \cite{Cminaffrank2}.
\begin{defn}
A simple module $\L(\bs f)\in\Ob\hat\cat$ is an \emph{affinization} of $V(\overline\mu)\in\Ob\cat$ if $\wt(\bs f)=\overline\mu$. Two affinizations of $V(\overline\mu)$ are \emph{equivalent} if they are isomorphic as $\uqg$-modules.
\end{defn} 
Our first observation is that $\L(\bs f)$ and $\L(\bs f^\dag)$, c.f. (\ref{dagdef}), are equivalent affinizations of $\wt(\bs f)$.
\begin{prop}\label{sameclass}
For all $\bs f\in\mc R$, $\L(\bs f)$ and $\L(\bs f^\dag)$ are isomorphic as $\uqg$-modules.
\end{prop}
\begin{proof}
For any $V\in \Ob\cat$, $(V^*)^\cinv\cong V$ as $\uqg$-modules, where $(V^*)^\cinv$ is the pull-back via the Cartan involution $\cinv$ of $V^*$. This is clear since $S(k_i) = k_i^{-1}= \cinv(k_i)$, so $\chi(V) = \chi((V^*)^\cinv)$. 
Consequently, the result follows from Corollary \ref{flipcor}.
\end{proof}

%\subsection{The partial order on affinizations}
Recall that an element $X$ of a partially ordered set $(P,\prec)$ is said to be \emph{minimal} if there is no $Y\in P$ such that $Y\prec X$, and is said to be \emph{least} if $Y\succeq X$ for all $Y\in P$. A partially ordered set has at most one least element. If a least element does exist then it is the unique minimal element.

For each $\mu \in \h^*$, there is a partial order on the equivalence classes of affinizations of $V(\overline\mu)$, defined as follows. Let $\L(\bs f),\L(\bs f')\in\Ob\hat \cat$ be two affinizations of $V(\overline\mu)$.  We say that the class of $\L(\bs f)$ weakly precedes that of $\L(\bs f')$ if and only if for all $\alpha\in Q^+$ either
\begin{enumerate}[(i)]
\item $\dim(\L(\bs f)_{\overline{\mu-\alpha}}) \leq \dim(\L(\bs f')_{\overline{\mu-\alpha}})$, or
\item there exists an $\beta\in Q^+$ such that $\beta\leq \alpha$ and 
$\dim(\L(\bs f)_{\overline{\mu-\beta}}) < \dim(\L(\bs f')_{\overline{\mu-\beta}})$.
\end{enumerate}
%\begin{rem} 
This partial order is given in terms of the dimensions of $\uqg$-weight spaces, but it could equivalently have been defined in terms of multiplicities of $\uqg$-module composition factors.
%, as in \cite{Cminaffrank2}.
%\end{rem}

\begin{defn}
A \emph{minimal (resp. least) affinization} of $V(\overline\mu)\in \Ob\cat$ is an equivalence class of affinizations of $V(\overline\mu)$ which is minimal (resp. least) with respect to this partial order. 
By a slight overloading we say also that any representative of such a class is a minimal (resp. least) affinization.
\end{defn}

\begin{defn}
A \emph{Kirillov-Reshetikhin} module is any least affinization of $V(\overline{\mu\omega_k})$, $\mu\in \C$, $k\in I$.   
\end{defn}
\begin{prop}
The module $\L(\bs f)$ is Kirillov-Reshetikhin if and only if $f_j(u) = \delta_{jk} q_k^\mu\frac{1-q_k^{-\mu-1} au}{1-q_k^{\mu-1} au}$ for some $a\in \Cx$.
\end{prop}
\begin{proof}
If $\bs f$ is of this form, then all non-highest weights of $\L(\bs f)$ are dominated by $\overline{\mu\omega_k-\alpha_k}$. If $f_k(u)$ is a string then $\dim(L)_{V(\overline{\mu\omega_k-\alpha_k})}=1$; otherwise $\dim(L)_{V(\overline{\mu\omega_k-\alpha_k})}>1$, c.f. Theorem \ref{sl2thm}. \end{proof}

Note that the least affinization of the irreducible $\uqslt$-module $V(\overline{\mu\omega_1})$, $\mu\in\C$, is an evaluation module, $V(\mu)_a$ (in the notation of \S\ref{sl2sec}).

\bigskip

In view of Proposition \ref{anacprop}, minimal (resp. least) affinizations with generic highest weights are limits (or rather analytic continuations) of minimal (resp. least) affinizations of finite-dimensional modules. Namely, let $\lambda = \sum_{i\in I}\lambda_i\omega_i$. Let $I = J \sqcup K$. Fix $\lambda_j$, $j\in J$, to be equal to given nonnegative integers. Let \be\nn\widetilde{\chi}(\lambda) := \overline\lambda^{-1} \chi(\L(\bs f))\ee be the normalized character of a least affinization $\L(\bs f)$ of $V(\overline\lambda)$. 
\begin{cor}\label{accor}
There exists a limit 
\be\nn \lim_{\substack{\lambda_k\to \8\\ k\in K}} \widetilde{\chi}(\lambda),\ee
where $\lambda_k$ run over $\N$, and this limit is equal to $\widetilde{\chi}(\lambda)$ with generic $\lambda_k\in \C$, $k\in K$. 
\end{cor}
\begin{proof} This follows from Proposition \ref{anacprop}.\end{proof}

Note that in particular there exists a limit of the normalized characters of Kirillov-Reshetikhin modules and it is equal to the normalized character of the Kirillov-Reshetikhin module with generic nontrivial component. Compare \cite{HJ}. 

In fact, in a similar way, there exists an analytic continuation of $\chi_q(\L(\bs f))$ with $\lambda_k\in \Z_{\geq 0}$, $k\in K$, 
which is equal to $\chi_q(\L(\bs f))$ with generic $\lambda_k\in \C$, $k\in K$.

\subsection{Classification of minimal affinizations in types $\mathrm{ABCFG}$}\label{ssthm}
For the remainder of this section we suppose $\g$ is of type $\mathrm{ABCF}$ or $\mathrm G$. 
We pick a straight labelling of the Dynkin diagram in which $B_{ij}\neq 0$ only if $|i-j|\leq 1$. 

In these cases the following theorem, which is the main result of \S\ref{sec:minaff}, shows that every simple object $V\in \Ob\cat$ has a least affinization. 
\begin{thm}\label{lathm}
Given $\lambda\in \h^*\setminus\{0\}$, an affinization $\L(\bs f)$ of $V(\overline\lambda)$ is least if and only if there is an $c\in \Cx$ and an $\eps\in \{+1,-1\}$ such that, for each $i\in I$,
\be f_i(u) = q^{\eps \lambda_i B_{ii}/2} \frac{1- c_i q^{-\eps \lambda_i B_{ii}} u}{1-c_i u},\nn \ee
where $c_1=c$ and $c_{i+1}q^{-\eps\lambda_{i+1} B_{i+1,i+1}} = c_iq^{-\eps B_{i,i+1}}$ for each $1\leq i< \rank(\g)$.
\end{thm}
The rest of \S5 is devoted to proving this theorem. After some preliminary lemmas in \S\ref{prelem}, in \S\ref{twonodesec} we treat the case in which $\lambda$ has support at the two end nodes of the Dynkin diagram; in \S\ref{threenodesec} we treat the case in which $\lambda$ has support at the two end nodes and one other node. Finally, the theorem is proved in \S\ref{gencase}. 

Let $\tau_a:\uqgh \to \uqgh$, $a\in \Cx$, be the automorphism defined by its action on generators according to \be \tau_a(x_{i,r}^\pm) = a^{\pm r} x_{i,r}^\pm,\qquad \tau_a(\phi^\pm_{i,\pm m}) = a^{\pm r} \phi^\pm_{i,\pm m}, \qquad \tau_{a} (k_i^{\pm 1}) = k_i^{\pm 1};\label{taudef}\nn\ee
then we have
\be \tau_a^*(\L(\bs f)) \cong \L( t_a(\bs f) ), \ee
where $t_a:\mc R\to \mc R$ is defined by $(t_a f_i)(u) := f_i(au)$.

It follows from Theorem \ref{lathm} that if $\L(\bs f)$ and $\L(\bs f')$ are least affinizations of $V(\overline\lambda)$ then there exists an $a\in \Cx$ such that either $\bs f = t_a(\bs f')$ or $\bs f =  t_a(\bs f'^\dag)$, c.f. (\ref{dagdef}).
 
Our strategy for identifying least affinizations relies on computing ``the top'' part of the $q$-character, in the following sense. For each $0\leq M \leq N=\rank(\g)$, define 
\be W_M := \left\{ - \sum_{k=1}^M \al_{i_k}: j\neq k \implies i_j\neq i_k\right\}.\nn\ee
We write $\chi_q(\L(\bs f))|_M$ for the $q$-character of $\L(\bs f)$ truncated to include only those terms whose weights lie in $\wt(\bs f) \overline{W_M}$. 
In certain cases, we shall compute $\chi_q(\L(\bs f))|_{M}$ for each $0\leq M \leq N$. 
That is, informally speaking, we shall consider all weights that can be reached from the highest weight by lowering at most once in each simple direction. 
As it turns out, this is sufficient to distinguish least affinizations from others.

\subsection{Preliminary lemmas}\label{prelem}
%For every $J \subseteq I$ let $\g_J$ be the corresponding diagram subalgebra of $\g$, and let $q_J := q^{\min_{j\in J}r_j}$. 

Let $\hat U_J$ be the subalgebra generated by $(x_{i,r}^+)_{i\in J, r\in \Z}$,  $(x_{i,r}^-)_{i\in J, r\in \Z}$, $(h_{i,r})_{i\in J, r\in \Z\setminus \{0\}}$ and $(k^{\pm 1}_i)_{i\in J}$ subject to the relations (\ref{hxpm}). 
Given a rational $\ell$-weight $\bs f=(f_{i}(u))_{i\in I}$, we write $\bs f_J$ for the $\ell$-weight $(f_{i}(u))_{i\in J}$ of $\hat U_J$, and $\L(\bs f_{J})$ for the irreducible $\hat U_J$-module with highest $\ell$-weight $\bs f_{J}$. 

%We shall use the shorthand $\hat U_J$ for $U_{q_J}(\mc L \g_J)$. Thus, $\hat U_J$ is the algebra generated by $(x_{i,r}^+)_{i\in J, r\in \Z}$,  $(x_{i,r}^-)_{i\in J, r\in \Z}$, $(h_{i,r})_{i\in J, r\in \Z\setminus \{0\}}$ and $(k^{\pm 1}_i)_{i\in J}$ subject to the relations (\ref{hxpm}). 
%Given an $\ell$-weight $\bs f=(f^\pm_{i,r})_{i\in I,r\in \Z}$, we write $\bs f_J$ for the $\ell$-weight $(f^\pm_i)_{i\in J,r\in \Z}$ of $\hat U_J$, and $\L(\bs f_{J})$ for the irreducible $\hat U_J$-module with highest $\ell$-weight $\bs f_{J}$. 

Similarly, let $U_J$ be the subalgebra generated by $(x_{i,0}^+)_{i\in J}$, $(x_{i,0}^-)_{i\in J}$ and $(k^{\pm 1}_i)_{i \in J}$. Given a weight $\rho\in \overline{\h^*}$ of $\uqg$ we write $\rho_J$ be for weight $(\rho_i)_{i\in J}$ of $U_J$, and $V(\rho_J)$ for the irreducible $U_J$-module with highest weight $\rho_J$. 

\begin{lem}\label{factorlem}
Let $I_1,I_2,\dots,I_k$ be subdiagrams of $I$ such that the corresponding diagram subalgebras $\g_{I_1},\g_{I_2},\dots, \g_{I_k}$ of $\g$ are simple and pairwise commuting. Let $L(\bs f)\in \Ob\hat\cat$ with highest $\ell$-weight vector $v$.  Then
\be \left( \hat U_{I_1}  \oplus \dots \oplus  \hat U_{I_k} \right)\on v \cong \L(\bs f_{I_1}) \otimes \dots \otimes \L(\bs f_{I_k}). \nn\ee
\end{lem}
\begin{proof}
Let $k=1$. Suppose $w\in \hat U_{I_1} \on v$ is a singular vector with respect to $\hat U_{I_1}$. Then, on weight grounds, $w$ is a singular vector with respect to $\uqgh$. Therefore, since $\L(\bs f)$ is irreducible, $w$ is proportional to $v$. Hence $\hat U_{I_1}\on v$ is irreducible.  

For general $k$ the lemma follows by the mutual commutativity of the $\g_{I_k}$. 
\end{proof}

\begin{lem}[The restriction lemma]\label{reslem} 
Let $J\subseteq I$. An affinization $\L(\bs f)\in\Ob\hat\cat$ of $V(\overline\mu)$ 
is least only if the simple $\hat U_J$-module through a highest $\ell$-weight vector $v$ of $\L(\bs f)$ is a least affinization of $V(\overline\mu_J)$. 
\end{lem}
\begin{proof}
Suppose there is a $J\subseteq I$ such that $\hat U_J\on v$ is not least.
Then there is another affinization of $V(\overline \mu_J)$, say $\L(\bs s_J)$, whose equivalence class does not weakly succeed that of $\L(\bs f_J)$. It follows that the class of $\L(\bs g)$ does not weakly succeed that of $\L(\bs f)$, where the rational $\ell$-weight $\bs g= (g_i(u))_{i\in I, r\in \Z_{\geq 0}}$ of $\uqgh$ is given by 
\be \nn g_i(u) := \begin{cases} s_i(u) & \text{if } i \in J, \\ f_i(u) & \text{if } i\in I \setminus J. \end{cases}\ee  
\end{proof}
\begin{cor}\label{sl2mincor}
A $\uqgh$-module $\L(\bs f)\in\Ob\hat\cat$ is a least affinization only if $f_i(u)$ is a string -- c.f. (\ref{Sdef}) -- for each $i\in I$. 
\end{cor}
\begin{proof}
We use the restriction lemma with $J=\{i\}$. Let $v$ be a highest $\ell$-weight vector of $\L(\bs f)$. It follows from the results of \S\ref{sl2sec} that $\left(\hat U_{\{i\}}\on v\right)_{\wt(\bs f)\overline{\alpha_i}^{-1}}$ has dimension $1$ if $f_i(u)$ is a string and dimension $\geq 2$ otherwise. 
\end{proof}
In the following lemma, we use $\simeq$ to denote equality up to a multiplicative constant. 
%For rational functions $f(u), g(u)$, we say $f(u)\simeq g(u)$ if $f(u)/g(u)$ is a constant. Then, if $f(u) \simeq \frac{1-cu}{1-au}$ and $f(u)$ is a rational $\ell$-weight then this constant is determined up to a sign by the condition that $f(0)f(\8) = 1$.  
\begin{lem}[Expansion lemma]\label{rk1lem}
Suppose $\bs f$ is an $\ell$-weight of $V\in \Ob\hat\cat$. Suppose $v\in V_{\bs f}$ and $i\in I$ are such that $x_{i,r}^+\on v=0$ for all $r\in \Z$. 
\begin{enumerate}[(i)]
\item If $f_i(u) \simeq \frac{1-cu}{1-au}$ with $a\neq c$, then $$x_{i,s}^+\left(x_{i,r}^- \on v  -a^r v_a\right) = 0\quad\text{for all $s\in \Z$,}$$ for some $\ell$-weight vector  $v_a\in V_{\bs f \bs A_{i,a}^{-1}}$. 
In particular, $v_a \equiv x_{i,0}^- \on v$ modulo $\bigcap_{r\in \Z}\ker x^+_{i,r}$.
\item If $f_i(u) \simeq 
%\sqrt{\frac{a}{c}}
\frac{1-cu}{1-au} 
%\sqrt{\frac{b}{d}}
\frac{1-du}{1-bu}$, with $a\neq b$ and $\{a,b\} \cap \{c,d\} = \emptyset$,
then $$x_{i,s}^+\left(x_{i,r}^- \on v - (a^r v_a + b^r v_b)\right) =0\quad\text{ for all $s\in \Z$,}$$ for some $\ell$-weight vectors $v_a\in V_{\bs f \bs A_{i,a}^{-1}}$ and $v_b\in V_{\bs f \bs A_{i,b}^{-1}}$. \\
In particular $v_a\equiv \frac{bx_{i,0}^- - x_{i,1}^-}{b-a}\on v$ and $v_b\equiv \frac{ax_{i,0}^- - x_{i,1}^-}{a-b} \on v$ modulo  $\bigcap_{r\in \Z}\ker x^+_{i,r}$.
\item If $f_i(u) \simeq 
%\sqrt{\frac{a}{c}}
\frac{1-cu}{1-au} 
%\sqrt{\frac{a}{d}}
\frac{1-du}{1-au}$, with $a\notin \{c,d\}$, then
$$x_{i,s}^+\left(x_{i,r}^- \on v - (a^r v_a + ra^{r} v'_a)\right) =0\quad\text{ for all $s\in \Z$,}$$ for some linearly independent vectors $v_a,v_a'\in V_{\bs f \bs A_{i,a}^{-1}}$.\\
In particular $v_a\equiv x_{i,0}^-\on v$ and $v_a'\equiv (a^{-1} x_{i,1}^- - x_{i,0}^-)\on v$ modulo  
$\bigcap_{r\in \Z}\ker(x_{i,r}^+)$.
\end{enumerate}
\end{lem}
\begin{proof}
Given Corollary \ref{taucor}, it is enough to show that statements (i), (ii) and (iii) hold in the  irreducible $\hat U_{\{i\}}$-module $W$ whose highest weight is $v$. So in the rest of this proof, we work in $W$. 

Let $\varphi_{i,s} := \phi^+_{i,s}- \phi^-_{i,s}$. 

In case (i) we have $\varphi_{i,s} \on v = \alpha a^s %\sqrt{\frac{a}{c}} 
(1-\frac{c}{a})$ for some $\alpha\in \Cx$ and thus, for all  $r,s\in \Z$,
\be (\varphi_{i,r+s} -a^r \varphi_{i,s}) \on v=0,\quad\text{and hence}\quad
 x_{i,s}^+(x_{i,r}^-\on v - a^r w ) = 0,\nn\ee
where $w_a=x_{i,0}^-\on v\in W$. 

In case (ii) we shall show that for all $r,s\in \Z$,  
\be\nn x_{i,s}^+\left(x_{i,r}^-\on v - a^r w_a  - b^r w_b\right)=0,\ee where $w_a=  \frac{bx_{i,0}^- - x_{i,1}^-}{b-a}\on v$ and $w_b=\frac{ax_{i,0}^- - x_{i,1}^-}{a-b} \on v$. This follows from
\be( (a-b) \varphi_{i,r+s} + (a^r b- b^r a) \varphi_{i,s} + (b^r-a^r) \varphi_{i,s+1} )\on v = 0,\nn\ee
which in turn holds because
\be \varphi_{i,r} \on v =  v  
%\sqrt{\frac a c \frac b d} 
\beta\left( a^{r-1} \frac{(a-c)(a-d)}{a-b} + b^{r-1} \frac{(b-c)(b-d)}{b-a}\right)\nn\ee
for some $\beta\in \Cx$. 

Finally in case (iii) we shall show that for all $r,s\in \Z$, $x_{i,s}^+(x_{i,r}^- v - (a^r w + ra^r w'))=0$ where $w:=x_{i,0}^- v$ and $w':= a^{-1} x_{i,1}^- v - x_{i,0}^- v$. For this it is enough to show that $x_{i,s}^+(x_{i,r}^- v - a^r x_{i,0}^- v - ra^{r-1} x_{i,1}^- v + ra^r x_{i,0}^- v)=0$ is singular, which is true if and only if
\be \left(\varphi_{i,r+s} -a^r \varphi_{i,s} - ra^{r-1} \varphi_{i,s+1} + ra^r\varphi_{i,s}\right)\on v=0 .\nn\ee
This is true, given that
\be \varphi_{i,r}\on v = v
%\sqrt{\frac{a^2}{cd}}
\gamma\left((r+1) a^r\left( 1- \frac c a - \frac d a + \frac{cd}{a^2}\right) + a^{r}\left(\frac c a +\frac d a -\frac{2cd}{a^2}\right)\right)\nn\ee
for some $\gamma\in \Cx$. 

Similar direct calculations show that, in each case (i), (ii) and (iii), the given vectors in $W$ have the $\ell$-weights claimed, and are not in $\bigcap_{r\in \Z} \ker(x_{i,r}^+)$ and hence are not zero.
\end{proof}

\subsection{The case of two nodes}\label{twonodesec}
We write
\be   \delta_{a,b} := \begin{cases} 1 & \text{if } a=b, \\ 0 & \text{otherwise}\end{cases} \qquad\text{and}\qquad 
\check\delta_{a,b} := \begin{cases} 0 & \text{if } a=b, \\ 1 & \text{otherwise.}\end{cases}\nn\ee 
Given any $a,c\in \Cx$, define:
\bea a_1:= a,\quad\text{and}\quad a_{i+1} &:=& a_i q^{-B_{i,i+1}}\text{ for each   }1\leq i\leq N-1,\\
 c_N:= c\quad\text{and}\quad c_{i-1} &:=& c_i q^{-B_{i-1,i}}\text{ for each } 2\leq i\leq N,\nn\eea
and then, for $0\leq K, S\leq N$, let 
\be \bs f_{K,S} := \bs f \cdot \left(\prod_{k=1}^K \A_{k,a_k}^{-1}\right) \cdot \left(\prod_{s=1}^S \A_{N+1-s,c_{N+1-s}}^{-1} \right). \nn\ee

\begin{prop}\label{twonodes}
Suppose $\bs f\in \mc R$ is of the form 
\be f_1(u) = q^{\mu B_{11}/2} \frac{1-q^{-\mu B_{11}}au}{1- au}, \quad
    f_N(u) = q^{\nu B_{NN}/2}\frac{1- q^{-\nu B_{NN}}cu}{1-  cu}, \label{f1n}\ee
for some $\mu,\nu\in \Cx$, and $f_j(u) = 1$ for all $1<j<N$. For each $0\leq M < N$, 
\be \chi_q(\L(\bs f))|_M = \sum_{K=0}^M \bs f_{K,M-K}. \quad
\label{wtk}\nn\ee
while
\be \chi_q(\L(\bs f))|_N = \sum_{K=1}^{N-1} \bs f_{K,N-K}\label{ptii} 
+\bs f_{N,0} \check\delta_{a_N,cq^{-\nu B_{NN}}}
+\bs f_{0,N} \check\delta_{c_1,aq^{-\mu B_{11}}} .\ee
\end{prop}
\begin{proof}
Let $v_{0,0}$ be a highest $\ell$-weight vector of $\L(\bs f)$ and define
\be v_{K,L} := (x_{K,0}^- x_{K-1,0}^-\dots x_{1,0}^-) (x_{N+1-L,0}^- \dots x_{N-1,0}^- x_{N,0}^-) v_{0,0}.\nn\ee
 
We shall first show by induction on $M$ that for all $0\leq M <N$, $v_{K,M-K}\in \L(\bs f)_{\bs f_{K,M-K}}$ and that $(v_{K,M-K})_{0\leq K\leq M}$ is a basis of $\bigoplus_{\lambda\in W_M}\L(\bs f)_\lambda$. This is true for $M=0$. Assume it is true for some $M<N-1$.
By definition of the $W_{M}$, $\bigoplus_{\lambda\in W_{M+1}}\L(\bs f)_\lambda$ is then spanned by the vectors
\be \left\{ x_{i,r}^- \on v_{K,M-K} : 0\leq K\leq M,  r\in \Z, K<i<N+1-M+K \right\}.\nn\ee

For all $K+1<i<N-M+K$, $x_{i,r}^- \on v_{K,L}=0$. 
%\be (\bs f_{K,L})_{N-K+L} = q^{-B_{N-K+L,N-K+L+1}} \frac{1-q^{B_{N-K+L,N-K+L+1}} c_{N-K+L+1} u}{1-q^{-B_{N-K+L,N-K+L+1}}c_{N-K+L+1} u}.\ee

Now we claim that for each $K$ such that $0\leq K\leq M$,
\be x_{N-M+K,r}^-\on v_{K,M-K}= c_{N-M+K}^r x_{N-M+K,0}^- \on v_{K,M-K} \in \L(\bs f)_{\bs f_{K,M-K+1}}\label{cl1}\ee
for all $r\in \Z$. Indeed, on weight grounds, $x_{i,s}^+x_{N-M+K,r}^-\on v_{K,M-K}=0$  for all $i\notin\{K,N-M+K\}$. For $M=0$, the claim then follows from Lemma \ref{rk1lem} part (i).  
For $M>0$, when $K> 0$ we need also the fact that (\ref{cl1}) holds at the previous step in the induction on $M$: $$x_{N-M+K,r}^-\on v_{K-1,M-K}= c_{N-M+K}^r x_{N-M+K,0}^- \on v_{K-1,M-K},$$ from which, since  $$x_{K,s}^+ x_{N-M+K,r}^-\on v_{K,M-K} = x_{N-M+K,r}^- x_{K,s}^+\on v_{K,M-K} = x_{N-M+K,r}^- \on v_{K-1,M-K} \lambda_s$$ for some coefficients $\lambda_s\in \C$, we have  $$x_{K,s}^+( x_{N-M+K,r}^-- c_{N-M+K}^r x_{N-M+K,0}^-) \on v_{K,M-K} =0.$$
Hence $x_{i,s}^+x_{N-M+K,r}^-\on v_{K,M-K}=0$  for all $i\neq K$. The claim then follows from Lemma \ref{rk1lem} part (i). 

By a similar argument, for each $K$ such that $0\leq K\leq M-1$,
\be x_{K+1,r}^-\on v_{K,M-K} = a_{K+1}^r x_{K+1,0}^- \on v_{K,M-K} \in \L(\bs f)_{\bs f_{K+1,M-K}}\nn\ee 
for all $r\in \Z$. 
Since $M<N-1$, $[x_{K+1,0}^-, x_{N-M+K,0}^-] = 0$. Hence, for all $0\leq K\leq M$,
\be x_{K+1,0}^- \on v_{K,M-K} = x_{N-M+K+1,0}^- \on v_{K+1,M-K-1} = v_{K+1,M-K}.\nn\ee This completes the inductive step.

%To see part (ii) first note that Proposition \ref{rk2prop} gives, for each rank 2 subdiagram $J=\{i,i+1\}$, the $q$-character of the $U_{q_J}(\widehat{\g_J})$-module $U_{q_J}(\widehat{\g_J})\on v_{N-2,i-1}$. Combining this with knowledge of $\chi_q(\L(\bs f))|_{N-1}$ from part (i)...
%This allows us to identify each monomial $\bs f_{K,L}$ of $\chi_q(\L(\bs f))|_{N-1}$ with  we know how these $\ell$-weight spaces are identified in $ $...

We turn to computing $\chi_q(L(\bs f))|_N$.
Note that for all $1<K<N$,
\be (\bs f_{K-1,N-K})_{K}(u)  =  
q^{-B_{K-1,K}} \frac{1-q^{B_{K-1,K}} a_{K-1} u} {1-q^{-B_{K-1,K}} a_{K-1} u} 
q^{-B_{K,K+1}} \frac{1-q^{B_{K,K+1}} c_{K+1} u} {1-q^{-B_{K,K+1}} c_{K+1} u}, \label{twostring}\ee
while
\bea\nn (\bs f_{0,N-1})_1(u) &=& 
q^{\mu B_{11}/2} \frac{1-q^{-\mu B_{11}}au}{1- au}
q^{-B_{12}} \frac{1-q^{B_{12}} c_{2} u} {1-q^{-B_{1,2}} c_{2} u},\\
         (\bs f_{N-1,0})_N(u) &=& 
q^{-B_{N-1,N}} \frac{1-q^{B_{N-1,N}} a_{N-1} u} {1-q^{-B_{N-1,N}} a_{N-1}u} q^{\nu B_{NN}/2}\frac{1- q^{-\nu B_{NN}}cu}{1-  cu}.\label{twostring2}\eea

Let us consider the generic case in which all of these rational functions are in lowest terms as written -- i.e. there are no cancellations -- and in which none have poles of second order.  
Now certainly,
\be \L(\bs f)|_N = \Span_{1\leq K\leq N,r\in \Z} x_{K,r}^- \on v_{K-1,N-K}.\nn\ee
Suppose $d\in\Cx\setminus \{a_K,c_{K}\}$. Then by the previous part, for all $j\neq K$, $\bs f_{K-1,N-K} \bs A_{K,d}^{-1} \bs A_{j,e}$ is not an $\ell$-weight of $\L(\bs f)|_{N-1}$ for any $e\in \Cx$. Hence, by Proposition \ref{stepprop}, for all $v\in \L(\bs f)_{\bs f_{K-1,N-K} \bs A_{K,d}^{-1}}$ we have $x_{j,r}^+\on v=0$ for all $j\neq K$, $r\in \Z$. So if $x_{K,r}^-\on v_{K-1,N-K}$ had a nonvanishing component $w$ in $\L(\bs f)_{\bs f_{K-1,N-K} \bs A_{K,d}^{-1}}$, then we would have to have $x_{K,s}^+\on w\neq 0$ for some $s\in \Z$ (otherwise $w$ would be singular). But, by Lemma \ref{rk1lem} part (ii), we know that modulo vectors in $\bigcap_{r\in \Z}\ker x_{K,r}^+$, every  $x_{K,r}^- \on v_{K-1,N-K}$ is in the span of the following two vectors:
\be (a_{K} x_{K,0}^- - x_{K,1}^-) \on v_{K-1,N-K} \in \L(\bs f)_{\bs f_{K-1,N+1-K}}, \nn\ee
\be (c_{K} x_{K,0}^- - x_{K,1}^-) \on v_{K-1,N-K} \in \L(\bs f)_{\bs f_{K,N-K}}.\nn\ee
Therefore  $\bs f_{K,N-K}$, $0\leq K\leq N$, and no others, are the $\ell$-weights of $\L(\bs f)|_N$. 

Now, the defining relations of $\uqgh$ include 
\be x_{K,1}^- x_{K+1,0}^- - q^{-B_{K,K+1}} x_{K+1,0}^- x_{K,1}^- 
 =  q^{-B_{K,K+1}} x_{K,0}^- x_{K+1,1}^- - x_{K+1,1}^- x_{K,0}^- \nn\ee
and we saw above that
\be x_{K,1}^- \on v_{K-1,N-1-K} = a_K x_{K,0}^- \on v_{K-1,N-1-K} \quad\text{and}\quad
    x_{K+1,1}^- \on v_{K-1,N-1-K} = c_{K+1} x_{K+1,0}^- \on v_{K-1,N-1-K}.\nn\ee
Hence, for any $1\leq K\leq N-1$,  the vector
\be u: = (c_K x_{K,0}^- -x_{K,1}^-) \on v_{K-1,N-K} = (c_K x_{K,0}^- -x_{K,1}^-) x_{K+1,0}^- \on v_{K-1,N-1-K} \nn\ee 
is equal to 
\be
  (a_{K+1} x_{K+1,0}^- -x_{K+1,1}^-) \on v_{K,N-1-K}= (a_{K+1} x_{K+1,0}^- -x_{K+1,1}^-) x_{K,0}^- \on v_{K-1,N-1-K}.\ee
And on $\ell$-weight grounds $x_{i,r}^+(\L(\bs f)_{\bs f_{K,N-K}}) =0$ for all $i\notin \{K,K+1\}$.
Hence every vector $v$ in the $\ell$-weight space $\L(\bs f)_{\bs f_{K,N-K}}$ is of the form
\be v = \lambda u  + w + y,\nn\ee
where 
\be w\in \left(\bigcap_{s\in \Z} \ker x_{K,s}^+\right)\cap \left(\Span_{r\in \Z} x_{K,r}^- \on v_{K-1,N-K}\right)\nn\ee
 and 
\be y\in \left(\bigcap_{s\in \Z} \ker x_{K+1,s}^+\right)\cap \left(\Span_{r\in \Z} x_{K+1,r}^- \on v_{K,N-1-K}\right).\nn\ee
Now by considering $0= [x^+_{K+1,s},x_{K,r}^- ] \on v_{K-1,N-K}$ we see that $w\in \bigcap_{s\in \Z} \ker x_{K+1,s}^+$. So $w=0$. Similarly $y=0$. Therefore $u$ spans $\L(\bs f)_{\bs f_{K,N-K}}$, and $\dim(\L(\bs f)_{\bs f_{K,N-K}}) = 1$, as required.

For $\bs f_{0,N}$ and $\bs f_{N,0}$ the logic is simpler because there is no need to identify vectors. On $\ell$-weight grounds and by Lemma \ref{rk1lem} part (ii) one finds  $\dim(\L(\bs f)_{\bs f_{0,N}}) = \dim(\L(\bs f)_{\bs f_{N,0}}) = 1$.

It remains to consider the exceptional cases in which cancellations or coincident poles occur in the functions $(\bs f_{K-1,N-K})_{K}(u)$ as written in (\ref{twostring}--\ref{twostring2}). In view of Lemma \ref{rk1lem}, the dimensions of the `outermost' $\ell$-weight spaces $\L(\bs f)_{\bs f_{0,N}}$ and $\L(\bs f)_{\bs f_{N,0}}$ drop to zero under exactly the conditions specified in (\ref{ptii}). 
On the other hand one finds that there are no conditions under which the dimensions of the $\ell$-weight spaces $\L(\bs f)_{\bs f_{K,N-K}}$, $1\leq K\leq N-1$, drop to zero. Suppose for example that a cancellation occurs in $(\bs f_{K-1,N-K})_K$ for some $1< K\leq N$: say  $q^{-B_{K-1,K}} a_{K-1} = q^{B_{K,K+1}} c_{K+1}$. That is, $a_{K+1} = c_{K+1}$. Then there is a double pole in $(f_{K,N-K-1})_{K+1}$ and $\dim(\L(\bs f)_{\bs f_{K+1,N-K-1}}) = 2$, which is correctly reflected in the expression (\ref{ptii}) since $\bs f_{K+1,N-K-1} = \bs f_{K, N-K}$ in this case. 
\end{proof}

\begin{cor}\label{twonc}
For all $\mu,\nu\in \C$, an affinization $\L(\bs f)$ of $V(\overline{\mu\omega_1 + \kappa\omega_N})$ is least if and only if 
\be f_1(u) = q^{\mu B_{11}/2} \frac{1-aq^{-\mu B_{11}}u}{1- au}, \quad
    f_N(u) = q^{\nu B_{NN}/2}\frac{1- cq^{-\nu B_{NN}}u}{1-  cu},\nn\ee
for some $a,c\in \C$ obeying at least one of the following equations:
\be aq^{-\sum_{i=1}^{N-1} B_{i,i+1}} = cq^{-\nu B_{NN}}, \qquad
    aq^{-\mu B_{11}} = cq^{-\sum_{i=1}^{N-1} B_{i,i+1}}. \label{minrs}\ee
\end{cor}
\begin{proof} 
If $\L(\bs f)$ obeys one of the two conditions (\ref{minrs}) then $L(\bs f^\dag)$ obeys the other, and by Proposition \ref{sameclass}, both define the same equivalence class.
The result is then immediate from Proposition \ref{twonodes} \end{proof}

\begin{rem}If $V(\overline{\mu\omega_1 + \kappa\omega_N})$ is finite-dimensional, i.e. $\mu,\kappa\in \Z_{\geq 0}$, at most one of the equations (\ref{minrs}) can hold. However, for infinite-dimensional modules they are not mutually exclusive. In type $A_2$, for example, the least affinization of $V=V(\overline{\mu\omega_1-\mu\omega_2})$ is the class of $\L(\bs f)$, where
\be f_1(u) = q^\mu\frac{1-aq^{-2\mu}u}{1-au}, \quad f_2(u) = q^{-\mu} \frac{1-aq^{\mu-1} u}{1-aq^{-\mu-1}u}, \qquad a\in \Cx. \nn\ee
This is an evaluation module, so $\L(\bs f)\cong V$ as $\uqg$-modules. By the usual Weyl character formula, all weight spaces of $V$ are one-dimensional. 
\end{rem}

\subsection{The case of three nodes}\label{threenodesec}
Now given any $b\in \Cx$, define
\bea b_j:= b,\quad\text{and}\quad b_{i-1} &:=& b_i q^{-B_{i-1,i}}\text{ for each } 1<i\leq j,\nn\\ b_{i+1} &:=& b_i q^{-B_{i,i+1}}\text { for each  } j\leq i<N,\nn\eea
and, for $0\leq K,L<j$ and $0\leq R,S < N+1-j$, 
\be \bs f_{K,L,R,S} := \bs f \A_{j,b}^{-1} 
\left(\prod_{k=1}^K \A_{i,a_i}^{-1}\right) 
\left(\prod_{\ell=1}^{L} \A_{j-\ell,b_{j-\ell}}^{-1} \right)
\left(\prod_{r=1}^R \A_{j+r,b_{j+r}}^{-1} \right)
\left(\prod_{s=1}^S \A_{N+1-s,c_{N+1-s}}^{-1} \right) . \nn\ee 

\begin{prop} \label{threenodes}
Suppose the rational $\ell$-weight $\bs f$ is of the form 
\be f_1(u) = q^{\mu B_{11}/2} \frac{1-q^{-\mu B_{11}}au}{1- au}, \quad
    f_j(u) = q^{\kappa B_{jj}/2}\frac{1- q^{-\kappa B_{jj}}bu}{1-  bu},\quad
    f_N(u) = q^{\nu B_{NN}/2}\frac{1- q^{-\nu B_{NN}}cu}{1-  cu},\nn\ee
for some $\mu,\kappa,\nu\in \Cx$, and $f_k(u)=1$ for all $k\notin \{1,j,N\}$. 
%Without loss of generality, suppose that $j-1\leq N-j$. (If not, relabel the Dynkin diagram, $i\leftrightarrow N+1-i$.)
Suppose further that 
\be  bq^{-\kappa B_{jj}} \in \{a_j, c_j \}.\label{xass}\ee
Then
\bea \chi_q(\L(\bs f))|_{\leq N} &=& %\underset{K+S=N}{\underbrace{\sum_{K=0}^{j-1}\sum_{S=0}^{N-j}}}  +
\sum_{K=0}^{j-1}\sum_{S=0}^{N-j} \bs f_{K,S} + 
\check\delta_{a_j,bq^{-\kappa B_{jj}}} 
\sum_{K=j}^{N-1}\sum_{S=0}^{N-K} \bs f_{K,S} + 
\check\delta_{c_j,bq^{-\kappa B_{jj}}} 
\sum_{S=N-j+1}^{N-1} \sum_{K=0}^{N-S} \bs f_{K,S}\nn\\ 
&&{} + \check\delta_{a_j,bq^{-\kappa B_{jj}}} \check\delta_{a_N,cq^{-\nu B_{NN}}} \bs f_{N,0} 
     + \check\delta_{c_j,bq^{-\kappa B_{jj}}} \check\delta_{c_1,aq^{-\mu B_{11}}} \bs f_{0,N} \nn\\
&&{} + \delta_{a_j,bq^{-\kappa B_{jj}}} \delta_{c_j,bq^{-\kappa B_{jj}}}
   \bs f_{j,N-j} \label{newterm}\\ 
%&&{} +
%\sum_{K=1}^{j-1} \sum_{L=0}^{j-1-K} \sum_{R=0}^{N-j-1} \sum_{S=1}^{N-j-R} \bs f_{K,L,R,S} \nn\\
&&{} +
\sum_{L=0}^{j-2} \sum_{K=0}^{j-1-L}  \sum_{R=0}^{N-j-1} \sum_{S=0}^{N-j-R} \bs f_{K,L,R,S} \nn\\
&&{}+
\check\delta_{b_1,aq^{-\mu B_{11}}} \sum_{R=0}^{N-j-1} \sum_{S=0}^{N-j-R} \bs f_{0,j-1,R,S}
+ \check\delta_{b_N,cq^{-\nu B_{NN}}} \sum_{L=0}^{j-2} \sum_{K=0}^{j-1-L}  \bs f_{K,L,N-j,0}\nn\\
&&{}+ \check\delta_{b_1,aq^{-\mu B_{11}}} \check\delta_{b_N,cq^{-\nu B_{NN}}} \bs f_{0,j-1,N-j,0}.\nn\eea
\end{prop}
\begin{proof}
This follows from arguments analogous to those used in the proof of Proposition \ref{twonodes}. The monomials $\bs f_{K,S}$ arise by lowering starting at the ends of the Dynkin diagram. The new possibility, as compared to Proposition \ref{twonodes}, is that one can also start to lower from node $j$, giving rise to monomials  $\bs f_{K,L,R,S}$.  
To understand the term $\delta_{a_j,bq^{-\kappa B_{jj}}} \delta_{c_j,bq^{-\kappa B_{jj}}}$ in (\ref{newterm}), note that
\be (\bs f_{j-1,0,0,N-j-1})_j(u) =  
q^{\kappa B_{jj}/2}\frac{1- q^{-\kappa B_{jj}}bu}{1-  bu} 
q^{-B_{j-1,j}} \frac{1-q^{B_{j-1,j}}a_{j-1}u}{1- a_{j} u}
q^{-B_{j,j+1}} \frac{1-q^{B_{j,j+1}} c_{j+1}u}{1-c_{j} u}.\label{maybethreestrings}\ee
When both $a_j=bq^{-\kappa B_{jj}}$ and $c_j=bq^{-\kappa B_{jj}}$ hold, only one of the denominators $1-a_ju=1-c_ju$ cancels, which still leaves one string ending at $a_j=c_j$. Thus $\bs f_{j,N-j}$ appears in $\chi_q(\L(\bs f))$ with multiplicity 1.  
\end{proof}
\begin{rem}
The condition (\ref{xass}) is included in order that, after cancellations, (\ref{maybethreestrings}) is the product of at most two strings. The formula given for $\chi_q(\L(\bs f))|_{\leq N}$ is actually still valid even if (\ref{xass}) is false. One can prove it using a generalization of the expansion lemma, Lemma \ref{rk1lem}, to the case of three strings. We do not need this result. \end{rem}

\begin{cor}\label{threenc}
Given $\mu,\kappa,\nu\in \C$, an affinization $\L(\bs f)$ of $V\left(\overline{\mu \omega_1 + \kappa \omega_j + \nu \omega_N}\right)$ is least if and only if  
\be f_1(u) = q^{\mu B_{11}/2} \frac{1-aq^{-\mu B_{11}}u}{1- au}, \quad
    f_j(u) = q^{\kappa B_{jj}/2}\frac{1- bq^{-\kappa B_{jj}}u}{1-  bu},\quad
    f_N(u) = q^{\nu B_{NN}/2}\frac{1- cq^{-\nu B_{NN}}u}{1-  cu},\nn\ee
for some $a,b,c\in \Cx$ such that either
\begin{enumerate}[(I)]
\item $aq^{-\sum_{i=1}^{j-1} B_{i,i+1}} = bq^{-\kappa B_{jj}}$  and 
$bq^{-\sum_{i=j}^{N-1} B_{i,i+1}} = cq^{-\nu B_{NN}}$, or
\item $aq^{-\mu B_{11}} = bq^{-\sum_{i=1}^{j-1} B_{i,i+1}}$ and 
   $bq^{-\kappa B_{jj}} = cq^{-\sum_{i=j}^{N-1} B_{i,i+1}}$.
\end{enumerate}
%If $\L(\bs f)$ and $\L(\bs f)$ are two affinizations of $V\left(\overline{\mu \omega_1 + \kappa \omega_j + \nu \omega_N}\right)$ such that the class of $\L(\bs f)$ least and the class of $\L(\bs s)$ is not, then
%\be \dim  
\end{cor}
\begin{proof}
If $\L(\bs f)$ obeys (I) then then $L(\bs f^\dag)$ obeys (II), so by Proposition \ref{sameclass} both (I) and (II) define the same equivalence class. It follows from Proposition \ref{threenodes} that all other equivalence classes of affinizations of  $V(\overline{\mu\omega_1 + \kappa\omega_N})$ strictly succeed this class. 
So it is indeed least. Note in particular that it strictly precedes the class defined by
\begin{enumerate}[(I)]
%\item $aq^{-\sum_{i=1}^{j-1} B_{i,i+1}} = bq^{-\kappa B_{jj}}$  and 
%$bq^{-\sum_{i=j}^{N-1} B_{i,i+1}} = cq^{-\nu B_{NN}}$;
%\item $aq^{-\mu B_{11}} = bq^{-\sum_{i=1}^{j-1} B_{i,i+1}}$ and 
%   $bq^{-\mu B_{jj}} = cq^{-\sum_{i=j}^{N-1} B_{i,i+1}}$;
\item[(III)] $aq^{-\sum_{i=1}^{j-1} B_{i,i+1}} = bq^{-\kappa B_{jj}}$ and 
   $bq^{-\mu B_{jj}} = cq^{-\sum_{i=j}^{N-1} B_{i,i+1}}$, or
\item[(IV)] $aq^{-\mu B_{11}} = bq^{-\sum_{i=1}^{j-1} B_{i,i+1}}$ and 
$bq^{-\sum_{i=j}^{N-1} B_{i,i+1}} = cq^{-\nu B_{NN}}$.
\end{enumerate}
By Proposition \ref{twonodes}, (I-IV) are the only affinizations that are least for the subdiagrams $\{1,\dots, j\}$ and $\{j,\dots, N\}$. The term (\ref{newterm}) vanishes in cases (I) and (II) but not in cases (III) and (IV). 
\end{proof}

%\begin{defn} For $1\leq k\leq K-1$, define $I^{(2)}_k := \{i_k,i_k+1,\dots,i_{k+1}\}$. We call the simple %diagram subalgebras $\g_{I^{(2)}_k}$ the \emph{subalgebras of $\lambda$-rank 2}.
%For $1\leq k\leq K-2$, define $I^{(3)}_k:= \{i_k,i_k+1,\dots,i_{k+2}\}$. We call the simple diagram subalgebras $\g_{I^{(3)}_k}$ the \emph{subalgebras of $\lambda$-rank 3}.
%\end{defn}
\subsection{Proof of Theorem \ref{lathm}}\label{gencase}
First we restate Theorem \ref{lathm} in the following form. We pick and fix, for this subsection, a $\lambda\in \h^*\setminus\{0\}$. 
Let $i_1< i_2 <\dots< i_K$ be such that $\lambda= \sum_{k=1}^K b_{k} \omega_{i_k}$, $b_{k}\neq 0$. 
Then we must show that an affinization $\L(\bs f)$ of $V(\overline \lambda)$ is least if and only if 
\be f_{i_k}(u) = q^{b_k B_{i_ki_k}/2} \frac{1-  q^{- b_k B_{i_ki_k}} a_k u}{1-a_k u}, \qquad 1\leq k\leq K,\label{gf}\ee
and $f_i(u)=1$ for all $i\in I\setminus \{i_1,\dots,i_K\}$, where $a_k\in \Cx$, $1\leq k\leq K$, are such that either
\begin{enumerate}[(I)]
\item $a_{k+1}q^{- b_{k+1} B_{i_{k+1},i_{k+1}}} = a_kq^{-\sum_{i=i_k}^{i_{k+1}-1}B_{i,i+1}}$ for all $1\leq k<K$ or
\item $a_{k}q^{- b_{k} B_{i_{k},i_{k}}} = a_{k+1}q^{-\sum_{i=i_k}^{i_{k+1}-1}B_{i,i+1}}$ for all $1\leq k<K$.
\end{enumerate}
Now, in view of Corollaries \ref{twonc} and \ref{threenc}, this statement is equivalent to the following   proposition.

\begin{prop}\label{threenodeprop}
An affinization $\L(\bs f)$ of $V(\overline \lambda)$ is least if and only if 
\begin{enumerate}
\item $\L(\bs f_{\{i_k\}})$ is least for each $1\leq k\leq K$, and
\item $\L(\bs f_{\{i_k,i_k+1,\dots,i_{k+1}\}})$ is least for each $1\leq k\leq K-1$, and 
\item $\L(\bs f_{\{i_k,i_k+1,\dots,i_{k+2}\}})$ is least for each $1\leq k\leq K-2$.
\end{enumerate}
\end{prop}
\begin{proof}
The ``only if'' part follows from Lemma \ref{reslem} and Corollaries \ref{sl2mincor}, \ref{twonc} and \ref{threenc}. 

For the ``if'' part, suppose $\bs f$ is such that conditions (1--3) hold. 
Then $\bs f$ is of the form given in (\ref{gf}).

When $K=1$ for every weight $\overline\mu\neq \overline\lambda$ of $\L(\bs f)$, then $\overline\mu \leq \overline{\lambda-\alpha_{i_1}}$ and the result is clear. So suppose that $\lambda$ has support at $K\geq 2$ nodes.
Let $\L(\bs s)$, $\bs s\in \mc R$, be any affinization of $V(\overline \lambda)$ and $w$ a highest $\ell$-weight vector of $\L(\bs s)$. 
Define 
\bea A^{(2)} &:=& \left\{ k \in \{1,2,\dots,K-1\} : \hat U_{\{i_k,i_{k}+1,\dots, i_{k+1}\}} \on w \text{ is not least}\right\} \nn,\\
 A^{(3)} &:=& \left\{ k \in \{1,2,\dots,K-2\} : \hat U_{\{i_k,i_{k}+1,\dots, i_{k+2}\}} \on w \text{ is not least}\right\}. \nn\eea

We first make the following observation:
\be\text{If $A^{(2)}=A^{(3)}=\emptyset$ then $\L(\bs s)$ and $\L(\bs f)$ are isomorphic as $\uqg$-modules.} \label{claim}\ee
Indeed, by Lemma \ref{reslem} and Corollaries \ref{twonc} and \ref{threenc},  $A^{(2)}=A^{(3)}=\emptyset$ holds only if $\bs f$ and $\bs s$ are of the form given in (\ref{gf}). Then, as noted in  \S\ref{ssthm}, there exists an $a\in \Cx$ such that either $\bs f = t_a(\bs f')$ or $\bs f \cong  t_a(\bs f'^\dag)$, c.f. (\ref{dagdef}), and therefore (\ref{claim}) follows from Proposition \ref{sameclass}.

Now we consider the case that $A^{(2)}\neq \emptyset$ or $A^{(3)}\neq\emptyset$. We shall show that the class of $\L(\bs f)$ strictly precedes that of $\L(\bs s)$ in the partial order.
By Proposition \ref{twonodes}, for all $k\in A^{(2)}$
\be\dim\left(\L(\bs s)
_{\overline{\lambda-\alpha_{i_k}-\alpha_{i_{k}+1}-\ldots-\alpha_{i_{k+1}}}}\right) 
> \dim\left(\L(\bs f)
_{\overline{\lambda-\alpha_{i_k}-\alpha_{i_{k}+1}-\ldots-\alpha_{i_{k+1}}}}\right).\nn\ee
By Proposition \ref{threenodes}, for all $k\in A^{(3)}$ 
\be \dim\left(\L(\bs s)
_{\overline{\lambda-\alpha_{i_k}-\alpha_{i_{k}+1}-\ldots-\alpha_{i_{k+2}}}}\right) 
> \dim\left(\L(\bs f)
_{\overline{\lambda-\alpha_{i_k}-\alpha_{i_{k}+1}-\ldots-\alpha_{i_{k+2}}}}\right). \nn\ee
It remains to compare weight spaces with weights that are not dominated by weights of the form $\lambda-\alpha_{i_k}-\alpha_{i_{k}+1}-\ldots-\alpha_{i_{k+1}}$, $k\in A^{(2)}$, or $\lambda-\alpha_{i_k}-\alpha_{i_{k}+1}-\ldots-\alpha_{i_{k+2}}$, $k\in A^{(3)}$.
Let $\mu$ be any such weight. Then there exist  simple, pairwise commuting, diagram subalgebras $\g_1,\dots,\g_T$ of $\g$ with corresponding subdiagrams $I_1,\dots, I_T$, and elements $\alpha^{(t)}\in Q_{I_t}^+$ for each $1\leq t\leq T$, such that $\mu = \lambda - \sum_{t=1}^T \alpha^{(t)}$ and such that, for each $1\leq t\leq T$, $\{i_k,i_k+1,\dots i_{k+1}\}\not\subseteq I_t$ for all $k\in A^{(2)}$ and  $\{i_k, i_k+1,\dots,i_{k+2}\}\not\subseteq I_t$ for all $k\in A^{(3)}$. It follows from the observation (\ref{claim}) above that for each $1\leq t\leq T$, $\L(\bs s_{I_t})$  is isomorphic to $\L(\bs f_{I_t})$ as a $\uqg$-module. 
Therefore, by Lemma \ref{factorlem}, we have:
\bea \dim\left(\L(\bs s)_{\overline\mu}\right) &=& \dim\left(\L(\bs s_{I_1})_{\overline{\mu_{I_1}}} \otimes \dots \L(\bs s_{I_T})_{\overline{\mu_{I_T}}} \right)\nn\\ &=& \prod_{t=1}^T \dim\left( \L(\bs s_{I_t})_{\overline{\mu_{I_t}}} \right) = \prod_{t=1}^T \dim\left( \L(\bs f_{I_t})_{\overline{\mu_{I_t}}} \right)
\nn\\ &=&  \dim\left(\L(\bs f_{I_1})_{\overline{\mu_{I_1}}} \otimes \dots \L(\bs f_{I_T})_{\overline{\mu_{I_T}}} \right) = \dim\left(\L(\bs f)_{\overline{\mu}}\right).\nn\eea
Hence the class of $\L(\bs f)$ strictly precedes that of $\L(\bs s)$ in the partial order, as required.
\end{proof}

\section{Character conjectures}
In this section we give a series of three conjectures, of increasing generality, for the classical (i.e. $\uqg$-) character of certain irreducible representations in $\hat\cat$. 
Our main interest is in the least affinizations of Verma modules, and these provide our starting point.
Computer experiments, using the algorithm of \cite{FM}, suggest that their characters have a simple form, similar to the Weyl denominator.

\begin{conj}[Least affinization of the generic Verma module]\label{lagv} 
Suppose $\g$ is of type $\mathrm{ABCF}$ or $\mathrm G$. Let $\L(\bs f)\in\Ob\hat\cat$ be a least affinization of $V(\overline\lambda)$, where $\lambda=\sum_{i\in I} \lambda_i\omega_i$ with $\lambda_i\notin \Z$ for any $i\in I$. Then
\be\nn \chi(\L(\bs f)) = \overline{\lambda} \prod_{\alpha\in \Delta^+} \left(\frac{1}{1-\overline{\alpha}^{-1}}\right)^{\max\limits_{i\in I} \left<\omega^\vee_i,\alpha\right>}. \ee
\end{conj}
This conjecture is known to hold in at least two special cases:
%\begin{itemize}
%\item 
\begin{prop}\label{c1p} Conjecture \ref{lagv} is true in types $\mathrm A_n$, $n\in \Z_{\geq 1}$, and $\mathrm B_2$.\end{prop}
\begin{proof}In type A, least affinizations are evaluation modules. So the least affinization $\L(\bs f)$ of an irreducible Verma module $V$ is isomorphic to $V$ as a $\uqg$-module. The formula for $\chi(\L(\bs f))$ in Conjecture \ref{lagv} is therefore correct, because it agrees with the usual character formula for Verma modules, i.e. the Weyl denominator. (Note that in type A,  $\max_{i\in I}\left<\omega_i^\vee,\alpha\right> = 1$ for all $\alpha\in \Delta^+$.)

%\item 
In type $\mathrm B_2$, for all $k,\ell\in \Z_{\geq 0}$ the least affinization $\L(\bs f)$ of $V(\overline{k\omega_1+\ell\omega_2})$ has as its $\uqg$-module decomposition  \cite{Cminaffrank2}
\be\nn \L(\bs f)\cong \bigoplus_{i=0}^{\lfloor\ell/2\rfloor} V(\overline{k\omega_1 + (\ell-2i)\omega_2})\ee
(with $\alpha_1$ the long root). By analytic continuation, c.f. Corollary \ref{accor}, the least affinization $\L(\bs f)$ of the generic $\lambda\in \h^*$ has character
\be\nn\chi(\L(\bs f)) =  \sum_{i=0}^\8 \chi\left(V\left(\overline{\lambda-2i\omega_2}\right)\right) = \overline{\lambda} \sum_{i=0}^\8 \left(\overline{\omega_2}\right)^{-2i} \frac1{1-\overline\alpha_1^{-1}} \frac1{1-\overline\alpha_2^{-1}}
\frac1{1-\overline\alpha_1^{-1}\overline\alpha_2^{-1}} \frac1{1-\overline\alpha_1^{-1}\overline\alpha_2^{-2}}
\ee 
using the usual character formula for Verma modules. Since $2\omega_2 = \alpha_1+2\alpha_2$, on summing the geometric series one has
\be\nn\chi(\L(\bs f)) = \overline{\lambda} 
\frac1{1-\overline\alpha_1^{-1}} \frac1{1-\overline\alpha_2^{-1}}
\frac1{1-\overline\alpha_1^{-1}\overline\alpha_2^{-1}} \left(\frac1{1-\overline\alpha_1^{-1}\overline\alpha_2^{-2}}\right)^2.
\ee 
as in Conjecture \ref{lagv}.\end{proof}
%\end{itemize}

Conjecture \ref{lagv} is the special case $J=I$ of the following. 
\begin{conj}[Least affinization of the generic parabolic Verma module]\label{lagpvm} 
Suppose $\g$ is of type $\mathrm{ABCF}$ or $\mathrm G$. Suppose $\lambda\in \h^*$ has support on some subdiagram $J\subseteq I$ (i.e. $\lambda = \sum_{j\in J} \lambda_j \omega_j$) and that $\lambda_j\notin \Z$ for any $j\in J$. 
Let $\L(\bs f)\in\Ob\hat\cat$ be a least affinization of $V(\overline\lambda)$. Then
\be\nn \chi(\L(\bs f)) = \overline{\lambda} \prod_{\alpha\in \Delta^+} \left(\frac{1}{1-\overline{\alpha}^{-1}}\right)^{\max\limits_{j\in J} \left<\omega^\vee_j,\alpha\right>}. \ee
\end{conj}

In many situations this conjecture can be deduced from existing results for finite dimensional modules. 
\begin{prop}
Conjecture \ref{lagpvm} is true in the following cases. \begin{enumerate}
\item In type $\mathrm A_n$, $n \in \Z_{\geq 1}$. 
\item In type $\mathrm G_2$ for Kirillov-Reshetikhin modules associated to the long node.
\item In type $\mathrm B_n$ with $J\subseteq \{1,2,3\}$ (where $n$ is the short node). 
\end{enumerate}
In addition, it agrees with the conjectured $\uqg$-module decompositions given in \cite{HKOTY} for the Kirillov-Reshetikhin modules for the long node in type $\mathrm C_n$. 
\end{prop}
\begin{proof}
The arguments are as in the proof of Proposition \ref{c1p}. In type A, the given character formula is just the usual character of the generic parabolic Verma module in type A.  

In type $\mathrm G_2$ the least affinization of $V(\overline{k\omega_2})$, where $\alpha_2$ is the long root, is a Kirillov-Reshetikhin module whose $\uqg$-module decomposition is $\bigoplus_{r=0}^k V(\overline{r\omega_2})$; see \cite{ChariMoura}. Since $\omega_2 = 3\alpha_1+2\alpha_2$, the result follows as in case of $\mathrm B_2$ in the proof of Proposition \ref{c1p}.
Similarly, in type $\mathrm B_n$ with $J\subseteq \{1,2,3\}$, the result follows from $\uqg$-module decompositions that can be found in \cite{Moura}.  
\end{proof}

Conjecture \ref{lagpvm} can be generalized somewhat further, as follows. 
Given $X\subseteq I$, let $\overline X$ denote the smallest connected subset of $I$ such that $X\subseteq \overline X$. 
%Given a node $i\in I$, an \emph{arm of the Dynkin diagram starting at $i$} is a subdiagram $\overline{\{i, e\}}$ where $e\in I$ is an end node. 
For any $i\neq j\in I$, the subdiagram $\overline{\{i,j\}}$ admits a \emph{straight labelling} $(i=j_1,j_2,\dots,j_{K-1},j_{K}=j)$ of its nodes: that is, one in which $B_{j_k,j_\ell}< 0$ if and only if $|k-\ell|= 1$ (and in which $j_k\neq j_\ell$ if $k\neq \ell$). 
Given a rational $\ell$-weight $\bs f$, we may factor the $|I|$-tuple of rational functions $(f_i(u))_{i\in I}$ as follows. For each $i\in I$, 
\be f_i(u) = \prod_{a\in \Cx/q^\Z} f^{(a)}_i(u),\ee
where every pole and every zero of $f^{(a)}_i(u)$ lies in $aq^\Z$. 

\begin{conj}\label{charconj}
Suppose that for each $a\in \Cx/q^\Z$ and each $i\in I$ there is an $n_a(i)\in \Z_{\geq -1}$, an $X\in \C$ and an $r\in \Z$ such that 
\be\nn f_i^{(a)}(u)= X(1-aq^ru)^{n_a(i)}.\ee 
Letting   
\be S^{(a)} := \left\{ i \in I : n_a(i) = -1 \right\}, \quad 
    U^{(a)} := \left\{ i \in I : n_a(i) > 0 \right\}, \nn\ee 
suppose further that
\begin{enumerate}
%\item for every $i\neq j$ such that $\overline{\{i,j\}}\cap S^{(a)}=\{i,j\}$, $\overline{\{i,j\}}\cap U^{(a)}\neq \emptyset$;
\item for all $i,j\in S^{(a)}$ with $i\neq j$, we have $U^{(a)}\cap \overline{\{i,j\}} \neq \emptyset$; 
\item for every $i\in S^{(a)}$, $j\in U^{(a)}$ such that $\overline{\{i,j\}} \cap (S^{(a)}\cup U^{(a)}) = \{i,j\}$, there is a straight labelling $(i=j_1, j_2, j_3, \dots, j_{K}=j)$ of $\overline{\{i,j\}}$
such that
\be f_{i}^{(a)}(u) = \frac{X}{1-aq^ru},\quad f_{j}^{(a)}(u) = W \left(1-aq^{r-\sum_{t=1}^{K-1} B_{j_t,j_{t+1}}}\right)^n,\nn\ee
for some $r\in \Z$, $n\in \Z_{> 0}$ and $X,W\in \Cx$; 
\item for every $j\in U^{(a)}$,  \be n_a(j)\geq |N_a(j)|-1 \quad\text{where}\quad
 N_a(j) := \left\{ i\in S^{(a)}: \overline{\{i,j\}} \cap (S^{(a)}\cup U^{(a)}) = \{i,j\} \right\}.\nn\ee
\end{enumerate}
Then 
\be \chi(\L(\bs f)) = \wt(\bs f) \prod_{a\in \Cx/q^\Z} \chi^{(a)},\quad \chi^{(a)} = \prod_{\alpha\in \Delta^+} \left(\frac 1 {1-\overline{\alpha}^{-1}}\right)^{
             \max\left(0,\sum\limits_{i\in S^{(a)}} \left<\omega^\vee_i,\alpha\right> - \sum\limits_{j\in U^{(a)}} n_a(j) \left<\omega^\vee_j, \alpha\right> \right)}.\label{charconjform}\ee
\end{conj}
Note that condition (3) is redundant except when the node $j$ is trivalent.\\ 
To see that Conjecture \ref{charconj} does entail Conjecture \ref{lagpvm}, let us give the following.
\begin{proof}[Proof of Conjecture \ref{lagpvm} assuming Conjecture \ref{charconj}]
Let $i_1< i_2 <\dots< i_K$ be the nodes of $J$. Without loss of generality, suppose $\bs f$ obeys condition (I) in Theorem \ref{lathm}. (If not, reverse the ordering of the Dynkin diagram.) 
Now we apply Conjecture \ref{charconj}. For each $1\leq k\leq K$ there is exactly one $a\in \Cx/q^\Z$ such that $i_k\in S^{(a)}$; and for this $a$, $U^{(a)} = \{i_{k+1}\}$ when $k<K$, while $U^{(a)}= \emptyset$ when $k=K$. Thus Conjecture \ref{charconj} implies
\be\nn \chi(\L(\bs f)) = \overline{\lambda} \prod_{k=1}^{K}\prod_{\alpha\in\Delta^+} \left(\frac{1}{1-\overline{\alpha}^{-1}}\right)^{\max\left(0,\left<\omega^\vee_{i_{k}},\alpha\right> -\left<\omega^\vee_{i_{k+1}}, \alpha\right>\right)}, 
%\nn\\&=& \overline{\lambda} \prod_{\alpha\in\Delta^+} \left(\frac{1}{1-\overline{\alpha}^{-1}}\right)^{\sum_{k=1}^K\max\left(0,\left<\omega^\vee_{i_{k}},\alpha\right> -\left<\omega^\vee_{i_{k+1}}, \alpha\right>\right)} 
\ee
where for convenience we define $\omega^\vee_{i_{K+1}}:=0$.
The result follows provided we can show that
\be\nn\sum_{k=1}^K\max\left(0,\left<\omega^\vee_{i_{k}},\alpha\right> -\left<\omega^\vee_{i_{k+1}}, \alpha\right>\right)= \max_{1\leq k\leq K} \left<\omega^\vee_{i_k},\alpha\right>.\ee
And indeed, this equality is a consequence of the following statement, which can be seen by case-by-case inspection: Let $\g$ be of type $\mathrm{ABCF}$ or $\mathrm G$ and rank $N$, and pick a straight labelling of the nodes of the Dynkin diagram; then for any positive root $\alpha$, the $N$-tuple $\left(\left<\omega^\vee_i,\alpha\right>\right)_{1\leq i\leq N}$ is unimodal, i.e. there is a $k$ such that
\be\nn \left<\omega^\vee_1,\alpha\right> \leq \left<\omega^\vee_2,\alpha\right> \leq \dots \leq %\left<\omega^\vee_{k-1},\alpha\right>\leq 
\left<\omega^\vee_k,\alpha\right> 
%\geq \left<\omega^\vee_{k+1},\alpha\right>
\geq \dots \geq \left<\omega^\vee_{N-1},\alpha\right> \geq \left<\omega^\vee_N,\alpha\right>.\nn\ee
As an obvious consequence, all sub-tuples are also unimodal. 
\end{proof}
%\begin{rem}
For Kirillov-Reshetikhin modules on end nodes in types $\mathrm{E_6, E_7, E_8}$, Conjecture \ref{charconj} can be matched against the conjectured $\uqg$-module decompositions of \cite{HKOTY}.
%\end{rem}

It is known \cite{CPminaffireg} that in types D and E the classification of minimal affinizations becomes more subtle in  \emph{irregular} cases: that is, for highest weights orthogonal to the simple root associated to the trivalent node. All other highest weights are called \emph{regular}. We believe that Conjecture \ref{charconj} applies in particular to minimal affinizations of parabolic Verma modules whose highest weights are regular.

Conjecture \ref{charconj} also applies to certain other modules, as the following examples illustrate.

\subsection*{Type $\mathrm A_4$.} \begin{tikzpicture}[baseline=-5,scale=.8] \draw[draw=white,double=black,very thick] (1,0) -- ++(2,0) -- ++(0:1) ;    
\filldraw[fill=white] (1,0) circle (1mm) node [below] {$1$};    
\filldraw[fill=white] (3,0) circle (1mm) node [below] {$3$};    
\filldraw[fill=white] (2,0) circle (1mm)  node [below] {$2$};    
\filldraw[fill=white] (3,0)++(0:1) circle (1mm)  node [below] {$4$};;    
%\filldraw[fill=white] (3,0)++(-60:1) circle (1mm) node [right] {$5$};;    
\end{tikzpicture}
\begin{itemize}\item If     
\be f^{(a)}_1(u) = \frac{1}{1-au},\quad f^{(a)}_2(u) = 1-aq^3u,\quad f^{(a)}_3(u)= \frac{1}{1-au},\quad f^{(a)}_4(u) = 1\nn\ee
then 
\be \chi^{(a)} = \frac{1}{1- \overline{\alpha_1}^{-1}}\frac{1}{1- \overline{\alpha_3}^{-1}}
               \frac{1}{1- \overline{\alpha_3}^{-1} \overline{\alpha_4}^{-1}}
   \frac{1}{1- \overline{\alpha_1}^{-1} \overline{\alpha_2}^{-1} \overline{\alpha_3}^{-1}}
\frac{1}{1- \overline{\alpha_1}^{-1} \overline{\alpha_2}^{-1} \overline{\alpha_3}^{-1}
             \overline{\alpha_4}^{-1}}.\nn\ee
\item If 
\be f^{(a)}_1(u) = 1-aq^3u,\quad f^{(a)}_2(u) = \frac 1{1-au},\quad f^{(a)}_3(u)= 1-aq^3u,\quad f^{(a)}_4(u) = 1\nn\ee
then 
\be \nn\chi^{(a)} = \frac{1}{1- \overline{\alpha_2}^{-1}}.\ee
\end{itemize}
\subsection*{Type $\mathrm F_4$}
\begin{tikzpicture}[baseline =-5,scale=.8] 
\draw[thick] (1,0)  -- (2,0);
\draw[double,thick] (2,0) -- (3,0);
\draw[thick] (3,0) -- (4,0);
\draw (2.4,.2) -- (2.6,0) -- (2.4,-.2);
\filldraw[fill=white] (1,0) circle (1mm) node [below] {$1$};    
\filldraw[fill=white] (3,0) circle (1mm) node [below] {$3$};    
\filldraw[fill=white] (2,0) circle (1mm)  node [below] {$2$};    
\filldraw[fill=white] (3,0)++(0:1) circle (1mm)  node [below] {$4$};;    
%\filldraw[fill=white] (3,0)++(-60:1) circle (1mm) node [right] {$5$};;    
\end{tikzpicture}
\begin{itemize}\item If     
\be f^{(a)}_1(u) = 1,\quad f^{(a)}_2(u) = 1-aq^5u,\quad f^{(a)}_3(u)= \frac{1}{1-au},\quad f^{(a)}_4(u) = 1-aq^3u\nn\ee
then 
\be \nn\chi^{(a)} = \frac{1}{1- \overline{\alpha_3}^{-1}}
               \frac{1}{1- \overline{\alpha_2}^{-1} \overline{\alpha_3}^{-2}}
   \frac{1}{1- \overline{\alpha_1}^{-1} \overline{\alpha_2}^{-1} \overline{\alpha_3}^{-2}}
.\ee
\item If\be f^{(a)}_1(u) = 1-aq^7u,\quad f^{(a)}_2(u) = 1,\quad f^{(a)}_3(u)= \frac{1}{1-au},\quad f^{(a)}_4(u) = 1-aq^3u\nn\ee
then 
\bea \chi^{(a)} &=& \frac{1}{1- \overline{\alpha_3}^{-1}}
               \frac{1}{1- \overline{\alpha_2}^{-1} \overline{\alpha_3}^{-1}}
  \left(\frac{1}{1- \overline{\alpha_2}^{-1} \overline{\alpha_3}^{-2}}\right)^2
 \frac{1}{1- \overline{\alpha_1}^{-1} \overline{\alpha_2}^{-1} \overline{\alpha_3}^{-2}} \frac{1}{1- \overline{\alpha_2}^{-1} \overline{\alpha_3}^{-2} \overline{\alpha_4}^{-1}}
\nn\\              
 &&\times\frac{1}{1- \overline{\alpha_1}^{-1} \overline{\alpha_2}^{-2} \overline{\alpha_3}^{-2}}              
\frac{1}{1- \overline{\alpha_1}^{-1} \overline{\alpha_2}^{-2} \overline{\alpha_3}^{-3} \overline{\alpha_4}^{-1}             }
\frac{1}{1- \overline{\alpha_1}^{-1} \overline{\alpha_2}^{-2} \overline{\alpha_3}^{-4} \overline{\alpha_4}^{-2}             }
\frac{1}{1- \overline{\alpha_1}^{-1} \overline{\alpha_2}^{-3} \overline{\alpha_3}^{-4} \overline{\alpha_4}^{-2}}             
.\nn\eea
\end{itemize}

%\begin{rem}
%\subsection*{Type $\mathrm D_5$}
To discuss the condition (3) in Conjecture \ref{charconj}, let us consider  type $\mathrm D_5$, \be\nn
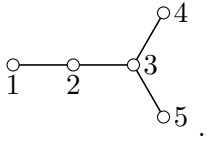
\begin{tikzpicture}[scale=.8] \draw[draw=white,double=black,very thick] (1,0) -- ++(2,0) -- ++(60:1) ++(60:-1) -- ++(-60:1) ;    
\filldraw[fill=white] (1,0) circle (1mm) node [below] {$1$};    
\filldraw[fill=white] (3,0) circle (1mm) node [right] {$3$};    
\filldraw[fill=white] (2,0) circle (1mm)  node [below] {$2$};    
\filldraw[fill=white] (3,0)++(60:1) circle (1mm)  node [right] {$4$};;    
\filldraw[fill=white] (3,0)++(-60:1) circle (1mm) node [right] {$5$};;    
\end{tikzpicture}.\ee
The reason for including the condition (3) is that computer experiments suggest that the formula for $\chi^{(a)}$ in (\ref{charconjform}) is not valid for
\be f^{(a)}_1(u) = 1-aq^4u,\quad f^{(a)}_2(u) = \frac{1}{1-aqu},\quad f^{(a)}_3(u) = 1-aq^4u,\quad f^{(a)}_4(u)= \frac{1}{1-aqu},\quad f^{(a)}_5(u) = \frac{1}{1-aqu}.\nn\ee
Here $n_a(3) = 1 <  3-1 = |N_a(3)| - 1$, so condition (3) is not satisfied.  
On the other hand, the following does fall within the scope of the conjecture:
\be f^{(a)}_1(u) = 1-aq^4u,\quad f^{(a)}_2(u) = \frac{1}{1-aqu},\quad f^{(a)}_3(u) = (1-aq^4u)^2,\quad f^{(a)}_4(u)= \frac{1}{1-aqu},\quad f^{(a)}_5(u) = \frac{1}{1-aqu};\nn\ee
and computer checks indicate that the formula for $\chi^{(a)}$ is valid in this case.

Finally, we find that in certain cases, the formula (\ref{charconjform}) appears to be valid even though condition (3) is not satisfied. For example if
\be f^{(a)}_1(u) = 1,\quad f^{(a)}_2(u) = \frac{1}{1-aqu},\quad f^{(a)}_3(u) = 1-aq^4u,\quad f^{(a)}_4(u)= \frac{1}{1-aqu},\quad f^{(a)}_5(u) = \frac{1}{1-aqu}\nn\ee
then the formula in (\ref{charconjform}) for $\chi^{(a)}$ does appear to hold.

\def\cprime{$'$}
\providecommand{\bysame}{\leavevmode\hbox to3em{\hrulefill}\thinspace}
\providecommand{\MR}{\relax\ifhmode\unskip\space\fi MR }
\providecommand{\MRhref}[2]{%
  \href{http://www.ams.org/mathscinet-getitem?mr=#1}{#2}
}
\providecommand{\href}[2]{#2}

\end{document}